\documentclass{amsart}    
\usepackage{graphicx}
\usepackage[percent]{overpic}
\usepackage[dvipsnames,usenames]{color}
\usepackage[colorlinks=true, urlcolor=NavyBlue, linkcolor=NavyBlue, citecolor=NavyBlue]{hyperref}
\usepackage{amssymb}
\usepackage{setspace}
\usepackage{pb-diagram}
\usepackage{listings}

\newtheorem{thm}{Theorem}    % Standard theorem environment
\newtheorem{lem}{Lemma}          % Lemma environment with numbering 
\newtheorem{cor}{Corollary}
\newtheorem{prop}{Proposition}

\theoremstyle{definition}
\newtheorem{defn}{Definition}    % Definition environment with 
\newtheorem{exam}{Example}
\newcommand{\inv}{^{-1}}

\newcommand{\del}{\partial}

\newcommand{\bbz}{{\mathbb{Z}}}
\newcommand{\bbq}{{\mathbb{Q}}}

\newcommand{\rank}{{\mathrm{rank}}}

\newtheorem{alg}{Algorithm}
\newcommand{\gn}[1]{G^{(#1)}}
\newcommand{\azn}[1]{\mathcal{A}^\bbz_{#1}(K)}
\newcommand{\an}[1]{\mathcal{A}_{#1}(K)}
\newcommand{\calk}{\mathcal{K}}
\newcommand{\bbk}{\mathbb{K}}

\title{On computing higher-order Alexander modules of knots}

\author{Peter D. Horn}
\address{Department of Mathematics\\Syracuse University\\215 Carnegie Building\\Syracuse, NY 13244-1150}
\email{pdhorn@syr.edu}
\urladdr{https://pdhorn.expressions.syr.edu/}
\thanks{The author was partially supported by NSF DMS-1258630.}

\begin{document}

\begin{abstract}
	Cochran defined the $n$th-order integral Alexander module of a knot in the three sphere as the first homology group of the knot's $(n+1)$th-iterated abelian cover.  The case $n=0$ gives the classical Alexander module (and polynomial).  After a localization, one can get a finitely presented module over a principal ideal domain, from which one can extract a higher-order Alexander polynomial.  We present an algorithm to compute the first-order Alexander module for any knot.  As applications, we show that these higher-order Alexander polynomials provide a better bound on the knot genus than does the classical Alexander polynomial, and that they detect mutation.  Included in this algorithm is a solution to the word problem in finitely presented $\bbz[\bbz]$-modules.
\end{abstract}

\maketitle

\section{Introduction}

Given a diagram for a knot $K$ in $S^3$, there is a well-known algorithm for producing the Wirtinger presentation of $\pi_1(S^3 \setminus K)$~\cite[Section 3.D]{Rolf}.  While the knot group is a complete knot invariant (see~\cite{Wald}), using these presentationsis an impractical tool for distinguishing knots.  Since any knot's exterior has $H_1(S^3 \setminus K ) \cong \bbz$, abelianizing the fundamental group is not at all a useful invariant for distinguising knots.  A logical next step would be to find an intermediate quotient, $Q(K)$, that is more discerning than $H_1$ and more tractable than $\pi_1$:  \[ \pi_1(S^3 \setminus K) \twoheadrightarrow Q(K) \twoheadrightarrow H_1(S^3\setminus K) \]  Or better yet, one may hope to find a module $M(K)$ on which the group $Q(K)$-acts.  With this strategy in mind, Cochran defined the higher-order Alexander modules of a knot~\cite{C:NonCom}.  The \emph{derived series of a group $G$} is defined by $\gn{0} = G$ and $\gn{n+1} = \left[\gn{n},\gn{n}\right]$, where the square brackets denote commutators.  Throughout this paper we are interested in knot groups $G = \pi_1(S^3 \setminus K)$.  Put simply in algebraic terms, the $n$th-order integral Alexander module of $K$ is $\gn{n+1}/\gn{n+2}$ as a right $\bbz\left[G/\gn{n+1}\right]$-module (the action is conjugation in the group).  For $n=0$, this is the classical Alexander module, which has a topological interpretation as the first homology of the universal abelian cover of the knot complement.  The higher-order Alexander modules we consider here are different from those defined before in Cochran-Orr-Teichner~\cite[Sections 2 and 3]{COT1}; the Cochran-Orr-Teichner modules are defined using a localization stemming from a coefficient system on a $4$-manifold in which the knot $K$ is slice.  The modules considered here only depend on the knot $K$.

Let $\azn{n}$ denote Cochran's $n$th-order integral Alexander module.  In ~\cite{C:NonCom}, Cochran constructs a localization of the coefficient ring and considers the $n$th-order (localized) Alexander module $\an{n}$.  A precise definition will follow in Section~\ref{sec:defs}.  As in the classical case, one can extract an integer-valued invariant $\delta_n(K)$, which is the degree of an ``$n$th-order Alexander polynomial.''  These degrees give lower bounds for the \emph{genus of a knot}, which is the minimal genus of all orientable surfaces in $S^3$ whose boundary is the given knot.

\begin{thm}[Theorems 7.1 and Corollary 9.2 of~\cite{C:NonCom}]\label{thm:coc}
	If $K$ is a knot in $S^3$ with non-trivial Alexander polynomial, then $\delta_0(K)\leq \delta_1(K)+1\leq \cdots \delta_n(K)+1\leq \cdots \leq 2\,\mathrm{genus}(K)$.  Furthermore, given any $n\geq 0$, there exists a knot where $\delta_n(K) < \delta_{n+1}(K)$.
\end{thm}

Thus, the higher $\delta_n$ give better lower bounds for the genus of a knot.  For the `$<$' part of the theorem, Cochran constructs knots by a satellite operation called `infection.'  Many positive results in the so-called `higher-order knot theory' have been proven using satellite operations, but satellite knots are not generic.  The higher-order degrees $\delta_n$ are the most alluring of all the higher-order knot invariants for two reasons: they are defined for all knots, and they `should be' algorithmically computable.  Since the $\delta_n$ were defined, several people have made successful computations by hand~\cite{LeiMax},~\cite{Holum}.  In a related line of inquiry, Harvey successfully computed degrees of higher-order Alexander polynomials of 3-manifolds~\cite[Sections 6 and 8]{Harvey:Thurston}.  Such calculations are tedious but always seem to work out with enough perseverance.

The purpose of this paper is to describe a practical algorithm to compute $\delta_1(K)$ for any knot $K$, answering question 13 of~\cite{FV:Survey}.  Using this algorithm, we compute $\delta_1$ for many low crossing knots.  All knots with ten crossings or fewer have $\delta_0 = 2\,\mathrm{genus}(K)$, and so we know $\delta_1$ by Theorem~\ref{thm:coc}.  It is easily seen from the definitions in Section~\ref{sec:defs} that if the classical Alexander polynomial of $K$ is trivial, then all the $\delta_n(K) = 0$.  The remaining knots with eleven crossings for which $0 < \delta_0 < 2\,\mathrm{genus}$ are $11_{n45}, 11_{n67}, 11_{n73}, 11_{n97}$, and $11_{n152}$.  We exhibit the utility of the algorithm to compute $\delta_1(11_{n67})$ and $\delta_1(12_{n293})$.

\begin{thm}\label{thm:11n67}
	$\delta_1(11_{n67}) = 3$ and $\delta_1(12_{n293}) = 3$.
\end{thm}

The classical Alexander polynomial of $11_{n67}$ is $2-5t+2t^2$, and its genus is $2$.  This is the simplest known example of a knot where $\delta_1$ is a better bound on the genus than $\delta_0$.  (The only simpler possible example is $11_{n45}$).  Since $3+1 = 2\cdot 2$ (as in Theorem~\ref{thm:coc}), we conclude that $\delta_n(11_{n67})=3$ for all $n\geq 1$.  We remark that the classical Alexander module of $11_{n67}$ is cyclic; it matches that of the stevedore knot, $6_1$, which has genus one.  The calculation of $\delta_1(11_{n67})$ took 11324 seconds on a machine with an Intel i5 2.5 GHz quad-core processor.

The knot $12_{n293}$ also has genus $2$, and its classical Alexander polynomial is $2-3t+2t^2$.

For each of $11_{n67}$ and $12_{n293}$, there exists another knot in the tables with isomorphic knot Floer homology; the computation was verified by Gridlink~\cite{Culler}.  These other knots are $11_{n97}$ and $12_{n519}$, respectively.  We were unable to find Wirtinger presentations for these knots that would allow our program to compute $\delta_1$ in a reasonable amount of time, which leaves the question: does knot Floer homology detect $\delta_1$?

We recall that the classical Alexander polynomials of mutant knots are equal (in fact, the stronger HOMFLYPT polynomial does not distinguish mutants)~\cite{LM}.  Wada showed that twisted Alexander polynomials distinguishes the mutant pair consisting of the Kinoshita-Terasaka and Conway knots ($11_{n42}$ and $11_{n34}$)~\cite[Section 6]{Wada}.  Noting that the higher-order and twisted Alexander polynomials have similar definitions, Friedl and Vidussi ask in their survey paper~\cite[Question 14]{FV:Survey}:
\begin{center}
	\emph{Do higher-order Alexander polynomials detect mutation?}
\end{center}

We are able to answer:

\begin{thm}\label{thm:DetectsMutants}
	The first-order Alexander polynomial can distinguish mutant knots.  In particular, $\delta_1(12_{n23}) = 3$ and $\delta_1(12_{n31}) = 5$.
\end{thm}

The proof of this result follows from Examples~\ref{ex:12n23} and ~\ref{ex:12n31} and the discussion in Section~\ref{sec:Mutants}.

A computer program for computing $\delta_1$ is available at \url{https://pdhorn.expressions.syr.edu/software/}.

\section{Definitions}\label{sec:defs}

The definitions of higher-order Alexander modules are due to Cochran~\cite{C:NonCom}.

Let $K$ be a knot in $S^3$ and $G = \pi_1(S^3\setminus K)$.  Let $\Gamma_n = G/\gn{n+1}$, where $\gn{n+1}$ is the $(n+1)$th term of the derived series of $G$.  There is a right action of $\Gamma_n$ on $\gn{n+1}/\gn{n+2}$ by conjugation: \[ g*h = h\inv\,g\,h,\ \ \mbox{for all } h\in\Gamma_n,\ g\in \gn{n+1}/\gn{n+2} \]  Let $M_n$ denote the $(n+1)$th iterated (maximal) abelian cover of $S^3\setminus K$, so that $\pi_1(M_n) \cong \gn{n+1}$.  Then $H_1(M_n; \bbz)\cong \gn{n+1}/\gn{n+2}$, and this homology group has a right action by $\Gamma_n$ (viewed as either deck translations or conjugation).

\begin{defn}
	The \emph{$n$th-order integral Alexander module of $K$}, denoted $\azn{n}$, is the right $\bbz\Gamma_n$-module \[ \azn{n} := \gn{n+1}/\gn{n+2} \cong H_1(S^3\setminus K; \bbz\Gamma_n) \cong H_1(M_n; \bbz) \]
\end{defn}

The group $\Gamma_n$ is \emph{poly-(torsion-free) abelian} (PTFA), i.e. it admits a series of normal subgroups $\{e\} \lhd A_k \lhd A_{k-1} \lhd \cdots\lhd A_0 = \Gamma_n$ so that each $A_i / A_{i+1}$ is torsion-free abelian.  To see this, take $A_i = \gn{i}/\gn{n+1}$ and recall that $A_i/A_{i+1} = \gn{i}/\gn{i+1}$ is torsion-free abelian~\cite{Strebel}.  We will use the fact that $\bbz\Gamma_n$ imbeds in its classical quotient field (in general $\bbz\Gamma_n$ is a noncommutative ring, so its quotient field is a skew field):

\begin{prop}[Proposition 3.2 of~\cite{C:NonCom}, see also the Proposition of~\cite{Lewin}]
	If $\Gamma$ is PTFA, then $\bbq\Gamma$ (and hence $\bbz\Gamma$) is a right Ore domain, i.e. $\bbq\Gamma$ imbeds in its classical right ring of quotients $\calk$, which is a skew field.
\end{prop}

Let $\calk_n$ denote the quotient field of $\bbz\Gamma_n$.  There is an important intermediate ring $R_n$ ($\bbq\Gamma_n \subset R_n \subset \calk_n$) which will be convenient for our use.  For $n\geq 1$, the kernel of $\pi: G/\gn{n} \to G/\gn{1}\cong \bbz$ is just $\gn{1}/\gn{n}$.  After choosing a splitting $1\mapsto\mu$ of $\pi$, we get an isomorphism \[ \Gamma_n = G/\gn{n+1} \cong \gn{1}/\gn{n+1} \rtimes \bbz \]  Any element of $\Gamma_n$ has a unique expression as $\mu^i\,g$, where $g\in \gn{1}/\gn{n+1}$.  There is also a unique expression by writing powers of $\mu$ on the right: $\mu^i\,g = \widetilde{g}\,\mu^i$, where the $\widetilde{\ }$ denotes the action on $g$ by conjugating $i$ times by $\mu$.  Thus $\bbq\Gamma_n$ is canonically isomorphic (after choosing the splitting) to the skew Laurent polynomial ring $\left(\bbq\left[\gn{1}/\gn{n+1}\right]\right)\left[t^{\pm 1}\right]$.  In this skew polynomial ring, the coefficients are not commutative (unless $n=1$), and the coefficients do not commute with the variable (due to the semidirect product structure).

Let $\bbk_n$ denote the classical right (skew) field of quotients for $\bbq\left[\gn{1}/\gn{n+1}\right]$, and let \\$R_n = \bbz\Gamma_n\left(\bbz \left[\gn{1}/\gn{n+1}\right] \setminus \{0\} \right)\inv$.  After choosing a splitting $1 \mapsto \mu$, there is a canonical isomorphism $R_n \xrightarrow{\cong} \bbk_n\left[t^{\pm 1}\right]$.  Since $R_n$ is a localization of $\bbz\Gamma_n$ wherein some but not all elements are inverted, we have that $\bbz\Gamma_n \subset R_n\subset \calk_n$.

\begin{defn}
	The \emph{$n$th-order (localized) Alexander module of $K$} is $\an{n} = H_1(S^3 \setminus K; R_n)$.
\end{defn}

We now summarize several useful facts about $R_n$ and $\calk_n$:

\begin{thm}[Section 4 of ~\cite{C:NonCom}]\label{thm:facts}
	Suppose we have chosen a meridian (i.e. splitting) $1\mapsto \mu$ so that $R_n\cong\bbk_n\left[t^{\pm 1}\right]$. Then
	\begin{enumerate}
		\item $R_n$ is a right (and left) PID,
		\item $\bbk_n\left[t^{\pm 1}\right]$ has a well-defined degree function and a Euclidean algorithm,
		\item $\an{n}$ is a finitely-generated torsion module over $R_n$,
		\item $R_n$ is a flat left $\bbz\Gamma_n$-module, so that $\displaystyle \an{n}\cong \azn{n}\otimes_{\bbz\Gamma_n} R_n$, 
		\item $\calk_n$ is a flat $R_n$ module so $\an{n}\otimes_{R_n} \calk_n = H_1(S^3\setminus K; \calk_n)$
	\end{enumerate}
\end{thm}

The structure theorem for finitely-generated modules over commutative PIDs generalizes to the noncommutative case~\cite{Cohn}: if $M$ is a finitely generated torsion right $R$ module, where $R$ is a PID, then $M\cong R/e_1\,R \oplus \cdots \oplus R/e_k\,R$ where $e_i$ is a total divisor of $e_{i+1}$ (moreover, this determines the $e_i$ up to similarity).

If one views the $e_i$ as elements of $\bbk_n\left[t^{\pm 1}\right]$, one could define the higher-order Alexander polynomials of $K$ as the product of the $e_i$.  However, this is not a well-defined invariant of $K$, as the isomorphism $R_n\cong \bbk_n\left[t^{\pm 1}\right]$ depends on a choice of meridian.  However, the degree of this polynomial is a knot invariant.

\begin{defn}[Definition 5.3 of ~\cite{C:NonCom}]\label{def:delta}
	For any knot $K$ and any $n\geq 0$, the \emph{degree of the $n$th order Alexander polynomial}, denoted $\delta_n(K)$, has several equivalent definitions:
	\begin{enumerate}
		\item the sum of the degrees of the $e_i \in R_n \cong \bbk_n\left[t^{\pm 1}\right]$,
		\item the rank of $\an{n}$ as a module over $\bbk_n$,
		\item the rank of $\gn{n+1}/\gn{n+2}\otimes_{\bbz\Gamma_n} R_n$ as a module over $\bbz\left[\gn{1}/\gn{n+1}\right]$, or
		\item the rank of $\gn{n+1}/\gn{n+2}$ as a module over $\bbz\left[\gn{1}/\gn{n+1}\right]$
	\end{enumerate}
\end{defn}

\begin{exam}
	If $n=0$, $\Gamma_0 = G/\gn{1} \cong \bbz$.  In this case $\bbk_0 = \bbq$, and $R_0 \cong \bbq[t^{\pm 1}]$.  Then $\azn{0}$ is the classical Alexander module $\gn{1}/\gn{2}$ over $\bbz[t^{\pm 1}]$, and $\an{0}$ is the rational Alexander module $\an{0} = \gn{1}/\gn{2}\otimes_{\bbz[t^{\pm 1}]} \bbq[t^{\pm 1}]$.
\end{exam}

\begin{exam}
	Consider the case $n=1$, which is the main focus of this paper.  The group $\Gamma_1 = G/\gn{2}$ fits into an exact sequence \[ 0 \to \gn{1}/\gn{2} \to \Gamma_1 \to \bbz \to 0 \]  After choosing a preferred meridian, which will be identified with $t$, we see that any element of $\Gamma_1$ can be written uniquely as $t^i\,g$, where $g\in \gn{1}/\gn{2}$.  In this case, $\bbk_1$ is the quotient field of $\bbz\left[\gn{1}/\gn{2}\right]$, which is a commutative field!  Addition and multiplication in this field work -- at least on a symbolic level -- as they do in $\bbq$, but one must be especially careful not to divide by zero.  Section~\ref{sec:wp} addresses this concern.  While $\bbk_1$ is commutative, the coefficients do not commute with the variable in $R_1 \cong \bbk_1[t^{\pm 1}]$.  For example, $t\, g = \mu g\mu\inv t$ in the polynomial ring.
\end{exam}

\subsection{Outline of the algorithm to compute $\delta_1$}\label{sec:quickoutline}

To compute $\delta_1(K)$, we will take a presentation matrix $M$ for $\an{1}$ as a module over $\bbk_1[t^{\pm 1}]$ and make it upper-triangular.  The sum of the degrees of the diagonal entries is $\delta_1(K)$.  Note that we do not need the divisibility condition on the diagonal entries as in the structure theorem (cf. Lemma~\ref{lem:utsuffices}).  The triangularization process requires only row operations (and some column swaps), which can be achieved by left-multiplication by elementary matrices over $\bbk_1[t^{\pm 1}]$.  The algorithm to upper-triangularize a matrix over $\bbk_1[t^{\pm 1}]$ is the same as over $\bbq[t^{\pm 1}]$.  First, multiply each row of $M$ by a suitable power of $t$ to make all entries \emph{honest} polynomials.  Pick the smallest degree entry and move it to the top-left entry by row and column swaps.  Use row operations to clear the first column; it may be necessary to perform a row swap if the degree of an entry in the $k,1$-position (where $k>1$) falls below the degree of the $1,1$-entry.  By doing these row operations, the matrix will eventually have a non-zero entry in $1,1$ and all zeros below that; this follows from items 2 and 3 of Theorem~\ref{thm:facts}.  The algorithm continues by moving on to the submatrix gotten by removing the first row and column.

The difficulty with implementing this algorithm lies in the coefficient field $\bbk_1$, which consists of formal quotients of formal $\bbz$-linear combinations of elements in the classical Alexander module, $\gn{1}/\gn{2}$, the group operation for which we will denote additively in this paragraph.  While this is a commutative field, arithmetic is lengthy.  For example, the sum \[ \frac{m\,a+n\,b}{o\,c} + \frac{p\,d}{q\,e+r\,f} = \frac{mq(a+e) + mr(a+f) + nq(b+e) + nr(n+f) + op(c+d) }{ oq(c+e) + or(e+f) } \] seems to have 5 terms in its numerator and two terms in its denominator.  But do any of the terms cancel?  Is the numerator zero?  Logically, the `collapsing problem' in the group ring $\bbz H$ is equivalent to the word problem in the group $H$.  However in practice, one would prefer to have a `normal form' for group elements in $H$ to speed up the ring operations in $\bbz H$.  This will be discussed further in Section~\ref{sec:wp}.

While $\bbk_1$ is a commutative field, its arithmetic seems to push the boundaries of the abilities of modern computers.    It seems to the author that an implementation to compute $\delta_n$ for $n\geq 2$ is hopeless with the current technology, since the fields $\bbk_n$ with $n\geq 2$ are noncommutative.

\section{Previous ad hoc computations}

For knots with very few crossings (five or fewer), one can calculate $\delta_1$ by hand.  Such a calculation is straightforward for two reasons: the size of the matrix you must work with is manageable, and the coefficient field $\bbk_1$ is simple, due to the simple structure of $\mathcal{A}_0$ for low crossing knots.  We present two examples to illustrate both the methods involved in the general algorithm and why previous $\delta_1$ computations have failed to give a general algorithm.  Compare the strategy in the next example with~\cite[Example 4.2]{Holum}.

\begin{exam}\label{ex:tref}
	Let $T$ denote the trefoil knot.  Using the standard torus knot diagram, one can calculate the Wirtinger presentation \[ G:= \pi_1\left( S^3 \setminus T \right) \cong \left\langle x_1, x_2, x_3|\ x_1x_3x_1\inv x_2\inv,\ x_1x_3x_2\inv x_3\inv \right\rangle \]  For reasons that will soon become clear, change the generating set to $a := x_1$, $y_1 := x_2x_1\inv$, $y_2 := x_3x_1\inv$.  Note that under the abelianization map, $a$ maps to a generator of $\bbz$ while $y_1$ and $y_2$ map to the identity.  We now have the presentation \[ \pi_1\left( S^3 \setminus T \right) \cong \left\langle a, y_1, y_2|\ ay_2a\inv y_1\inv, \ ay_2y_1\inv a\inv y_2\inv \right\rangle \]  We construct the matrix of Fox derivatives~\cite{Fox} whose $i,j$-entry is the derivative of the $i$th relation with respect to the $j$th generator; the Fox calculus is defined by \begin{itemize}
		\item $\frac{\del g_j}{\del g_i} = \delta_{ij}$, for each $i,j$
		\item $\frac{\del 1}{\del g_i} = 0$, for each $i$
		\item $\frac{\del uv}{\del{g_i}} = \frac{\del u}{\del g_i} + u\frac{\del v}{\del g_i}$, for each $i$ and for each $u,v$ in the group
	\end{itemize}
	where the $g_i$ are the generators of the group.  It follows that $\frac{\del u\inv}{\del g_i} = -u\inv\frac{\del u}{\del g_i}$ for all $u$ in the group.  The Fox matrix for the second presentation is \[\begin{pmatrix}
		1 - ay_2 a\inv & -1 & a\\
		1 - ay_2y_1\inv a\inv & -ay_2y_1\inv &a-1
	\end{pmatrix}\]  The entries of this matrix lie in the group ring $\bbz \pi_1$.  Upon writing the group elements in any quotient group $G$, one obtains a presentation matrix for $H_1\left(S^3 \setminus T,\  x_0;\  \bbz G\right)$, the first homology relative to a basepoint of the covering space whose group of covering transformations is $G$.
	
	Let us write the Fox matrix `over $\bbz$' by abelianizing the group elements: \[ \begin{pmatrix}
		0 & -1 & t \\
		0 & -t & t-1
	\end{pmatrix} \]  The first column has all zeros by the choice of generators $a, y_1, y_2$.  Recall that the Fox matrix is a presentation matrix for the first homology of a covering space relative to a basepoint; this, in effect, adds a free generator to $H_1\left(S^3 \setminus T; \bbz\left[t^{\pm 1}\right]\right)$, since in the long exact sequence for the pair $(S^3\setminus T, x_0)$, we have

	\[
	\begin{diagram}
	\node{0} \arrow{e}
		\node{ H_1\left(S^3 \setminus T; \bbz\left[t^{\pm 1}\right]\right) } \arrow{e,t}{} \arrow{s,b}{\cong}
	   \node{ H_1\left(S^3 \setminus T, x_0; \bbz\left[t^{\pm 1}\right]\right) } \arrow{s,r}{\cong} \arrow{e,t}{}
		\node{ H_0\left( x_0; \bbz\left[t^{\pm 1}\right]\right) } \arrow{s,b}{\cong}
	\\
	\node[2]{\mathrm{torsion}} \node{ \mathrm{free}\oplus\mathrm{torsion} } \node{\bbz\left[t^{\pm 1}\right]}
	\end{diagram}
	\]
	
	By exactness, the free part of $H_1\left(S^3 \setminus T, x_0; \bbz\left[t^{\pm 1}\right]\right)$ has rank one.
	
The free generator in this case is the one corresponding to the first column, and upon eliminating that column, we end up with a presentation for $H_1\left(S^3 \setminus T; \bbz \left[t^{\pm 1}\right]\right)$, i.e. the classical Alexander module of $T$: 	\[ \begin{pmatrix}
			-1 & t \\
			-t & t-1
		\end{pmatrix} \]  By the definition of $y_1,y_2$, $y_1,y_2$ may be viewed as generators of the classical Alexander module (they lie in the commutator subgroup), and the first and second columns of the above matrix correspond to these generators.  In other words, $\mathcal{A}^\bbz_0(T)$ is generated as a $\bbz\left[t^{\pm 1}\right]$-module by $y_1$ and $y_2$ with the relations $-y_1 + t*y_2 = 0$ and $-t*y_1 + (t-1)*y_2 = 0$.  One can actually `upper-triangularize' this presentation matrix (using only row operations, which amounts to adding multiples of the relations together) to: \[ \begin{pmatrix}
			1 & -t\\
			0 & t^2 - t + 1
		\end{pmatrix}\]  Thus $y_1 = t*y_2$ and $(t^2-t+1)*y_2 =0$ \emph{in the Alexander module} $\mathcal{A}^\bbz_0(T)$.  This description of $\mathcal{A}^\bbz_0(T)$ will allow us to understand the coefficients in $\bbk_1$.
		
		To compute $\delta_1(T)$, we need a presentation matrix for $\mathcal{A}_1(T) = H_1\left(S^3\setminus T; \mathcal{R}_1\right)$.  Since $\bbz\Gamma_1 \subset \mathcal{R}_1$, and $\Gamma_1$ is a quotient of $\pi_1$, such a presentation matrix can be derived from the Fox matrix above.  First, reduce the elements in the Fox matrix according to the quotient map $\pi_1\left(S^3 \setminus T \right) \twoheadrightarrow G/G^{(2)}=\Gamma_1$.  Recall from Section~\ref{sec:defs} that any element in $\Gamma_1$ has a unique expression (using the semidirect product) once we have chosen a meridian.  We have already chosen a distinguished meridian: $a$.  Let us write the Fox matrix with entries written in $G/\gn{2}\cong \gn{1}/\gn{2} \rtimes \bbz$, with $t$ generating $\bbz$ (we use $a$ when it conjugates an element of $\gn{1}/\gn{2}$, and $t$ when there is no conjugation): \[ \begin{pmatrix}
			1 - ay_2a\inv & -1 & t\\
			1 - ay_2a\inv ay_1\inv a\inv & -t\, y_2y_1\inv & t-1
		\end{pmatrix} \]  Again, this matrix presents $H_1\left(S^3 \setminus T,\  x_0;\  \bbz \Gamma_1\right)$.  Since $R_1$ is a flat $\bbz\Gamma_1$ module, so this matrix also presents $H_1\left(S^3 \setminus T,\  x_0;\  R_1\right)$ since $\bbz\Gamma_1 \hookrightarrow R_1$.  To get a presentation matrix for $H_1\left(S^3 \setminus T;\  R_1\right)$, we must eliminate one column as we did above for $\bbz\left[t^{\pm 1}\right]$.  Let $r_i$ denote the $i$th relation and $x_j$ the $j$th generator in the presentation for $G$.  By the fundamental formula for Fox derivatives~\cite[eq 2.3]{Fox}, \begin{equation}\label{eq:foxidentity}
			\frac{\del r_i}{\del x_j}\left( x_j -1 \right) = \sum_{j'\neq j} \frac{\del r_i}{\del x_{j'}}\left( x_{j'} -1 \right) 
		\end{equation} This implies that the $j$th column of the Fox matrix is a $R_1$-linear combination of the other columns as long as $x_{j} \in \gn{1}$ and $x_{j} \neq 1 \in \bbz \Gamma_1$.  As long as we can find a generator $y_i \in \gn{1}$ which is not in $\gn{2}$, we can (perform column operations according to Equation~\ref{eq:foxidentity} until the column corresponding to $y_i$ has all zeros and) delete the column corresponding to $y_i$.  This results in a presentation matrix for $H_1(S^3 \setminus T; R_1)$, since $\rank_{R_1} H_1\left(S^3 \setminus T, x_0; R_1 \right) = 1$, which we will justify now.  The existence of the column of zeroes establishes that $\rank_{R_1} H_1\left(S^3 \setminus T, x_0; R_1 \right) > 0$.  One may choose a splitting $H_1\left(S^3 \setminus T, x_0; R_1 \right) \cong R_1^d \oplus \mathrm{torsion}$ and consider the map $(R_1)^d \to H_0(x_0; R_1) \cong R_1$ induced by the right most map in the following portion of the long exact sequence for the pair $(S^3 \setminus T, x_0)$:

		\[
		\begin{diagram}
		\node{0} \arrow{e}
			\node{ H_1\left(S^3 \setminus T; R_1 \right) } \arrow{e,t}{} \arrow{s,b}{\cong}
		   \node{ H_1\left(S^3 \setminus T, x_0; R_1 \right) } \arrow{s,r}{\cong} \arrow{e,t}{}
			\node{ H_0\left( x_0; R_1 \right) } \arrow{s,b}{\cong}
		\\
		\node[2]{\mathrm{torsion}} \node{ \mathrm{free}\oplus\mathrm{torsion} } \node{R_1}
		\end{diagram}
		\] 
The map $(R_1)^d \to R_1$ is injective by the exactness of the sequence above.  That $d=1$ now follows from the rank-nullity theorem over a PID.
		
The upper-left triangular presentation matrix for $\mathcal{A}_0^\bbz(T)$ guarantees that $y_2 \neq 1 \in \gn{1}/\gn{2}$.  Thus, we remove the last column of the Fox matrix over $\bbz \Gamma_1$ (which corresponds to $y_2$) to obtain a square matrix, which we will call the \emph{metabelian Fox matrix}: \begin{equation}
			\begin{pmatrix}\label{eq:trefmat}
				1 - ay_2a\inv & -1 \\
				1 - ay_2a\inv ay_1\inv a\inv & -t\, y_2y_1\inv
			\end{pmatrix}
		\end{equation}
		
		The reader may have noticed that the second column may also be deleted (the element $y_1$ is nontrivial in $\gn{1}/\gn{2}$).  We chose the last column because fewer row operations are required to upper-triangularize matrix~\ref{eq:trefmat}.
		
		Since $\bbz\Gamma_1 \subset \mathcal{R}_1 \cong \bbk_1\left[t^{\pm 1}\right]$, matrix~\ref{eq:trefmat} is a presentation matrix for $\mathcal{A}_1(T)$, as discussed above.  We choose in our algorithm to use only row operations at this point because they are achieved by multiplications on the left by elementary matrices over $\bbk_1\left[t^{\pm 1}\right]$.  Column operations are equally valid but require the extra mental effort of remembering to do multiplication \emph{on the right} in the noncommutative ring $\bbk_1\left[t^{\pm 1}\right]$!
		
		We add $-t\ y_2y_1\inv$ times the first row to the second and switch columns: \[ \begin{pmatrix}
			-1 & 1-ay_2a\inv\\
			0& t\ \left(ay_2a\inv y_2y_1\inv - y_2y_1\inv\right) + 1 - ay_2a\inv ay_1\inv a\inv
		\end{pmatrix} \]
		
		At this point, adding the degrees of the diagonal polynomials gives $\delta_1(T)$ by Lemma~\ref{lem:utsuffices}, but we must be careful.  The $-1$ has degree zero.  Linear-looking polynomials such as the bottom-right entry do not necessarily have degree one; for example, $tx + y$ will have degree $1$ if and only if $y\neq 0$ and $x\neq 0 \in \bbk_1$.  Recall that $\bbk_1$ is the quotient field of $\bbz\left[\gn{1}/\gn{2}\right]$.  We can determine whether the coefficents are zero using the presentation matrix for $\mathcal{A}_0^\bbz(T)$.  The $t$-coefficient $ay_2a\inv y_2y_1\inv - y_2y_1\inv$ is zero if and only if $ay_2a\inv ay_1\inv a\inv = 1$ in $\gn{1}/\gn{2}$.  This equality in the group can be translated into an equality in the $\bbz\left[t^{\pm1}\right]$-module $\mathcal{A}_0^\bbz(T)$: $t*y_2 - t*y_1 = 0$?  Since $y_1 \neq y_2$ in $\mathcal{A}_0^\bbz(T)$, the $t$-coefficient in the linear-looking polynomial is nonzero.  We leave it to the reader to verify that the constant term is also nonzero.
		
		We compute $\delta_1(T) = \deg(-1) + \deg(t\ \left(ay_2a\inv y_2y_1\inv - y_2y_1\inv\right) + 1 - ay_2a\inv ay_1\inv a\inv) = 0 + 1 = 1$.
\end{exam}

Few choices were made in the computation of $\delta_1(T)$ (choice of preferred meridian, which column to delete).  These choices did not affect the end result, and one can easily program a computer how to make these choices.  The burden of the computation lies in putting the metabelian Fox matrix in upper-triangular form.  In general, many row (and perhaps column operations) are needed, and the $\bbk_1$-coefficients quickly become unruly.  While Example~\ref{ex:tref} was fairly simple, an outline for the general algorithm to compute $\delta_1$ can be gleaned:

\begin{alg}\label{alg:outline}
	\ 
	\begin{enumerate}
		\item Input: a diagram for a knot $K$.
		\item Compute a Wirtinger presentation for $\pi_1(S^3 \setminus K)$ and change the generators.
		\item Compute the general Fox matrix and abelian Fox matrix.
		\item Be able to recognize $0$ in the ring $\bbz\Gamma_1$ so row operations over $\bbk_1\left[t^{\pm1}\right]$ are possible.
		\item Decide which column(s) of the Fox matrix can be deleted, and delete one of them to arrive at the metabelian Fox Matrix.
		\item Put the metabelian Fox matrix in upper-triangular form using valid row operations over $\bbk_1\left[t^{\pm1}\right]$.
		\item Compute and add the degrees of the polynomials on the diagonal, giving $\delta_1(K)$.
	\end{enumerate}
\end{alg}

\begin{proof}
	We will describe an algorithm for each step and a lemma pertaining to steps 6 and 7.  Recall that Definition~\ref{def:delta} item 1 requires a \emph{diagonal} matrix that presents $\an{1}$ wherein the diagonal entries are total divisors of the subsequent diagonals.  The divisibility criterion is immaterial in the computation of the degrees of the diagonal polynomials.  Lemma~\ref{lem:utsuffices} guarantees that an upper-triangular, rather than diagonal, form of the presentation matrix suffices, which will save considerable time by avoiding many tedious row and column operations.
	
	Steps 1, 2 and 3 are straightfoward.  We use in our implementation a script written by an REU group at Columbia University that converts a knot diagram (PD Code) to a Wirtinger presentation~\cite{REU}.  Step 4 is the most difficult.  One needs to understand $\bbz\left[ \gn{1}/\gn{2}\right]$ well enough to know when an element is zero or nonzero; outside of that, addition, multiplication and inversion in $\bbk_1$ behave like the same operations in $\bbq$.  The proofs of Theorem~\ref{thm:swp} given in Sections~\ref{sec:ram},~\ref{sec:sw}, and~\ref{sec:gb} settle step 4.  Step 5 is then simple as in Example~\ref{ex:tref}: the column corresponding to one of the generators $y_i\in \gn{1}$ may be deleted if and only if $y_i \neq 0$ in $\gn{1}/\gn{2}$. Step 6 is settled by an adaptation to $\bbk_1\left[t^{\pm1}\right]$ of the algorithm to compute the canonical form of a module over a PID (cf. Section~\ref{sec:quickoutline}), although for the purpose of computing $\delta_1$, one only needs upper triangular and does not need the divisibility criterion as discussed above.  Finally, step 7 is straightforward (we use the degree function for Laurent polynomials, so $\deg(t + t^{-2}) = 3$ and $\deg(t^2 + t) = 1$).
	
\end{proof}

Several groups of undergraduates have attempted to implement an algorithm to compute $\delta_1$, and the sticking point tends to be step 4.  A strategy that works for many low crossing knots is to compute the abelian Fox matrix, which is a presentation for $\azn{0} \cong \gn{1}/\gn{2}$, and \emph{put it in upper triangular form}.  This works for the trefoil in Example~\ref{ex:tref}.  The problem is that `upper-triangularizing the abelian Fox matrix' is impossible in general, as we see now.

\begin{exam}\label{ex:nut}
	Let $K$ denote the knot in Figure~\ref{fig:nut}.  One may compute via Fox calculus or from a Seifert surface a presentation for $\azn{0}$ to be \[ P = \begin{pmatrix}
		3 - 3t & 2 - t\\
		1 - 2t & 3 - 3t
	\end{pmatrix} \]  The Alexander polynomial of $K$ is $7t^2 - 13t + 7$, an irreducible polynomial over $\bbz\left[t^{\pm 1}\right]$.
	
	$P$ can be put in upper-triangular form over $\bbq\left[t^{\pm1}\right]$, which is a PID, but this is impossible over $\bbz\left[t^{\pm1}\right]$, for suppose \[ P \sim \begin{pmatrix}
		a&b\\
		0&c
	\end{pmatrix} \] where $\sim$ denotes a sequence of invertible row and column operations over $\bbz\left[t^{\pm 1}\right]$. We may assume without loss of generality that $a = 1$, since $7t^2-13t+7$ is irreducible.
	
	Recall that the first elementary ideal of a module with an $r\times r$ presentation matrix is the ideal $\mathcal{E}_1$ generated by all $r-1 \times r-1$ minors of the matrix.  For a $2\times 2$ matrix such as $P$, $\mathcal{E}_1$ is the ideal generated by all of its entries.  This ideal is invariant under the operations $\sim$.  For our $P$, $\mathcal{E}_1 = \langle 2-t, 1-2t\rangle$, since $3-3t$ is the sum of the two off-diagonal entries.  If $P$ were upper-triangularizable as above, then $\mathcal{E}_1 = \bbz\left[t^{\pm 1}\right]$.  Then we could write some power of $t$ as $t^m = p(t)(t-2) + q(t)(2t -1)$ for some $p(t), q(t) \in \bbz[t]$.  Plugging in $t = 2$, we see that $2^m = q(2)*3$, but $2^m$ is not divisible by $3$, a contradiction.
	
	\begin{figure}[!ht]
		\begin{center}
			\begin{overpic}[scale=.8]{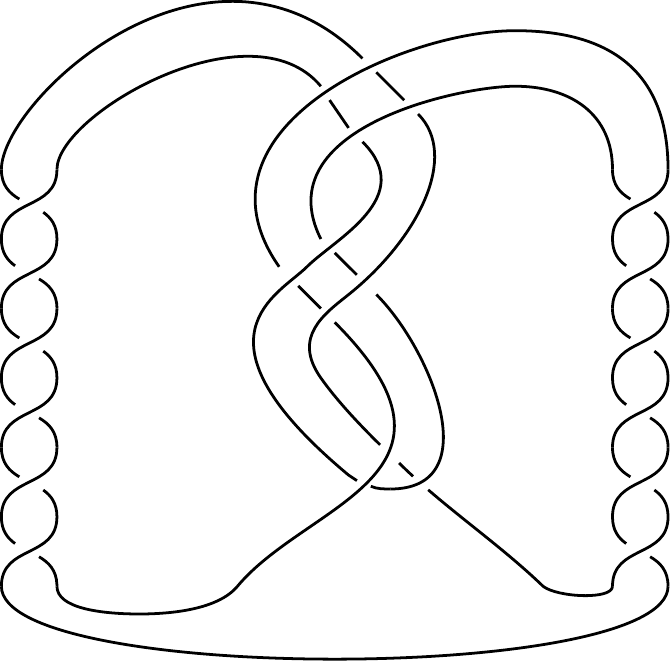}

			\end{overpic}
			\caption{A knot with non-upper-triangularizable abelian Fox matrix}
			\label{fig:nut}
		\end{center}
	\end{figure}
\end{exam}

We now present the lemma which is crucial in the proof of Algorithm~\ref{alg:outline}.

\begin{lem}\label{lem:utsuffices}
	Suppose $P$ is a square presentation matrix for $\an{n}$, or more generally for any torsion $\bbk_n\left[t^{\pm 1}\right]$-module.  If $P$ is upper-triangular with diagonal entries $d_1,\ldots,d_k$, then \[ \sum_{i=1}^k \deg d_i = \mathrm{rank}_{\bbk_n} \an{n} = \delta_n(K) \]
\end{lem}

\begin{proof}
	The claim follows essentially from the Euclidean algorithm.  Row (respectively, column) operations are realized by left (respectively, right) multiplication by elementary matrices.  Thus, to reduce an entry $c$ in the matrix by another entry $b$ of smaller degree using a row operation, we need to use the `left quotient' $c = qb + r$.  Right quotients will be used for column operations.  Since $\bbk_n\left[t^{\pm1}\right]$ is a PID, one can transform $P$ to a diagonal matrix using these operations.  Furthermore, it suffices to prove the lemma for $2\times 2$ matrices.  For, suppose the lemma has been established for $2\times 2$ matrices.  Given \[ P = \begin{pmatrix}
		a&b&c\\ 0&d&e\\ 0&0&f
	\end{pmatrix} \] we may change $P$ to \[ P \leadsto \begin{pmatrix}
		a'&0&c\\ 0&d'&e' \\ 0&0&f
	\end{pmatrix} \] by only disturbing the first two rows and columns.  This may in turn be changed to \[ \begin{pmatrix}
		a''&0&0\\ 0 & d' & e' \\ 0&0&f'
	\end{pmatrix} \] by extending the operations on the submatrix \[ \begin{pmatrix}
		a'&c\\ 0 &f
	\end{pmatrix} \] to the entire $3\times 3$ matrix.  One continues until $P$ has been completely diagonalized.  One continues by induction to prove the lemma for matrices of arbitrary size.
	
	It remains to establish the $2\times 2$ case.  Suppose \[ P = \begin{pmatrix}
		a & b\\ 0 & c
	\end{pmatrix} \]  Given an element $x \in \bbk_n\left[t^{\pm 1}\right]$, we denote its degree by using uppercase symbols, i.e. $X = \deg x$.  Returning our focus to $P$, we may perfome some row and column operations to guarantee that $B < A$ and $B < C$.  The ring $\bbk_n\left[t^{\pm 1}\right]$ has the Euclidean algorithm.  We compute the quotients $q_1,\ldots q_{n+2}$ and remainders $r_1\ldots r_{n+1}$:
	\begin{eqnarray*}
		c &=& q_1b+r_1\\
		b &=& q_2r_1 + r_2\\
		r_1 &=& q_3r_2 + r_3\\
		&\vdots&\\
		r_{n-1} &=& q_{n+1}r_{n} + r_{n+1}\\
		r_{n} &=& q_{n+2}r_{n+1} + 0
	\end{eqnarray*}
	
	Let $O_i$ for $i = 1,\ldots, n+2$ denote the row operation which adds $-q_i$ times row $\frac{3}{2}+\frac{1}{2}(-1)^i$ to the other row.  The first operation is \[ \begin{pmatrix}
		a&b\\ 0&c
	\end{pmatrix} \xrightarrow{O_1} \begin{pmatrix}
		a&b\\ -q_1a & r_1
	\end{pmatrix} \]  Note that $O_1$ takes $P$ out of the set of upper-triangular matrices, but the full sequence of the $O_i$ almosts takes $P$ back to upper-triangular.  To see this, define a recursion relation by $f_{-1} = 0$, $f_{0} = a$, and $f_{k} = f_{k-2} - q_kf_{k-1}$.  We see immediately that the degrees satisfy $F_i = Q_i + F_{i-1}$ for  $i\geq 1$, since $Q_i > 0$, and we see that \[ \begin{pmatrix}
		a&b\\ 0&c
	\end{pmatrix} \xrightarrow{O_1} \begin{pmatrix}
		f_0 & b\\ f_1 & r_1
	\end{pmatrix}	 \xrightarrow{O_2} \begin{pmatrix}
			f_2 & r_2\\ f_1 & r_1
		\end{pmatrix}	 \xrightarrow{O_3} \begin{pmatrix}
				f_2 & r_2\\ f_3 & r_3
			\end{pmatrix} 	 \xrightarrow{O_4}	 \begin{pmatrix}
					f_4 & r_4\\ f_3 & r_3
				\end{pmatrix} \cdots \] Since $r_{n+2} = 0$, this sequence of operations ends at either \[ 	 \begin{pmatrix}
						f_{n+1} & r_{n+1}\\ f_{n+2} & 0
					\end{pmatrix} \hspace{5mm}\mbox{or}\hspace{5mm} 	\begin{pmatrix}
								f_{n+2} & 0 \\ f_{n+1} & r_{n+1}
							\end{pmatrix} \] depending on the parity of $n$.  Let $S$ denote the operation which either swaps the two columns or swaps the columns and swaps the rows, so that \[ S(O_{n+2}\cdots O_1P) = \begin{pmatrix}
								r_{n+1} & f_{n+1}\\ 0 & f_{n+2}
							\end{pmatrix}\]
	is upper-triangular.  Perform a column operation to arrive at \[ P' = \begin{pmatrix}
								r_{n+1} & f'_{n+1}\\ 0 & f_{n+2}
							\end{pmatrix} \] where $F'_{n+1} < R_{n+1} < B$.
	Note that $R_{n+1} + F_{n+2} = R_{n+1} + Q_{n+2} + F_{n+1} = R_{n} + F_{n+1}$ by the definitions of the $r_i$, $q_i$, and $f_i$.  We apply this observation repeatedly to conclude that $R_{n+1} + F_{n+2} = R_1 + F_2$, but this is in turn $R_1 + Q_2 + Q_1  A = B + Q_1 + A = C + A$, which is the sum of the diagonal degrees of $P$.
	
	Thus, we have described a procedure that takes $P = \begin{pmatrix}
		a & b \\ 0 & c
	\end{pmatrix}$ to $P' = \begin{pmatrix}
		a' & b' \\ 0& c'
	\end{pmatrix}$ such that $A+C = A'+C'$ and $B> B'$. After repeatedly aplying this procedure, we end at a diagonal matrix ($b' = 0$) which presents the same module as $P$ and has the same diagonal degree sum.
\end{proof}

\section{The word problem in finitely presented $\bbz[\bbz]$-modules}\label{sec:wp}

Recall that $\bbk_1$ is the (commutative) quotient field of the ring $\bbz\left[\gn{1}/\gn{2}\right]$.  In doing the $\bbk_1\left[t^{\pm 1}\right]$-row operations to put the metabelian Fox matrix in upper-triangular form, it will be necessary to compute (left) polynomial quotients in the ring $\bbk_1\left[t^{\pm 1}\right]$.  Thus, one needs to invert elements of $\bbk_1$ and in particular, needs to know when an element in $\bbz\left[\gn{1}/\gn{2}\right]$ is zero so as not to divide by zero in $\bbk_1$!  This is the difficulty of step 4 of Algorithm~\ref{alg:outline}.

Let $H$ be any group, and consider its group ring $\bbz H$.  Given an element $p = n_1h_1 + \cdots + n_kh_k$ of $\bbz H$, one may determine whether or not $p=0$ by the algorithm: for $2 \leq i \leq k$, if $h_i = h_j$ for some $j < i$, combine the $i$th and $j$th terms by adding $n_j + n_i$.  After running this $O(k^2)$ algorithm, $p=0$ if and only if all remaining $n_i = 0$.

Suppose we have a presentation of the group $H$ with generators $x_i$ and relators $r_j$ (neither indexing set must be finite), and let $h$ be an arbitrary element of $H$, written as a word in the generators.  The problem of deciding whether $h$ is equal to the identity element (represented by the empty word) is known as the \emph{word problem} for this presentation of $H$.  We say this presentation of $H$ has \emph{solvable word problem} if there is an algorithm taking any element $h = h(x_i)$ of $H$ and deciding whether $h=1$ or not.  We refer the reader to Stillwell's enlightening survey on the word problem~\cite{Stillwell}.  While much effort has been made to understand the word problem for finitely presented groups, the case of infinitely presented groups has received less attention.

We are interested in the group $H = \gn{1}/\gn{2}$, where $G = \pi_1(S^3\setminus K)$.  In this case $H$ is an abelian group, and $H$ is finitely generated if and only if the (classical) Alexander polynomial $\Delta_K(t)$ is monic.  The word problem for finitely generated abelian groups is solvable by the structure theorem for finitely generated $\bbz$-modules, but sadly most knots do not have monic Alexander polynomial.  Recall that $H$ has the structure of a $\bbz[\bbz]$-module.  Even though $\bbz[\bbz]$ is not a PID (and so its modules have no a priori structure theorem), the word problem for such modules is solvable.

\begin{thm}\label{thm:swp}
	Let $H$ be an abelian group.  Suppose that $H$ has finite presentation $P$ as a module over $\bbz[\bbz]$. Then there exists a group presentation $P'$ for $H$ that has solvable word problem.
\end{thm}

Note that $P$ is not a presentation of $H$ as a group, though a group presentation can be obtained from $P$.  For each $\bbz[\bbz]$-module generator $x$, one considers the $\bbz$-many generators $t^i*x$.  Similarly, each $\bbz[\bbz]$-module relation gives rise to $\bbz$-many group relations.  These generators and relations give a group presentation $P'$ for $H$.

The author could not find this result in the literature and expects it to be `folklore.'  The special case when $H$ is the classical Alexander module of a knot in $S^3$ is well-known to knot theorists; see Sections~\ref{sec:ram} and~\ref{sec:sw}.  We present a full proof in Section~\ref{sec:gb}.  The proof for the general case leads to a more effective implementation of Algorithm~\ref{alg:outline} than the proofs in Sections~\ref{sec:ram} and~\ref{sec:sw}.

\begin{cor}
	Given a finite presentation $P$ for a $\bbz[\bbz]$-module $H$, the word problem in $\bbz H$ is solvable.
\end{cor}
\begin{proof}
	Given $p = n_1h_1 + \cdots + n_k h_k$.  One can determine whether $h_i = h_j$ since the word problem for $H$ is solvable, and by the argument above, one can collapse $p$ until it is no longer collapsible.
\end{proof}

We would like to remark that the $O(k^2)$ algorithm to collapse $p = n_1h_1 + \cdots + n_k h_k$ is far from optimal.  In putting the metabelian Fox matrix in upper-triangular form, many row operations may be required.  This may well involve inverting many elements in $\bbk_1$, and each time an inversion is done, we must check that we are not dividing by zero.  The expressions $p = n_1h_1 + \cdots + n_k h_k$ get longer as more row operations are performed.  An $O(k^2)$ operation done to many elements $p$ of increasing length is very time consuming.  Thus, a \emph{normal form} for elements of $H$ is preferred.  A normal form for elements in a group is a canonical expression which allows for a quick decision of whether $h = 1$.  For example, consider the polynomial multiplication $(1 + t + t^2)*(1-t-t^2)$, and imagine for a moment that the terms ($t^i$) are unrelated.  Multiplying the polynomials using distributivity yields $1*1 + -1*t - 1*t^2 + \cdots$.  One could then collapse the terms using the $O(k^2)$ algorithm described above.  Now, remember that $t^1*t^1 = t^2$ (i.e. $t^i$ is a normal form for a monomial).  Then one may quickly multiply $(1 + t + t^2)*(1-t-t^2)$ by grouping the terms in the product by degree, i.e. using the normal form: from the distributed product $1 - t - t^2 + \cdots$, one may rewrite the product (using a bit of memory) by scanning through the product only once.  Here is the full product $1-t-t^2+t-t^2-t^3 + t^2-t^3-t^4 = (1)1 + (-1+1)t + (-1-1+1)t^2 + (-1-1)t^3 + (-1)t^4$.  Assuming that integer addition is instantaneous, this method for collapsing polynomials is $O(k)$, where $k$ is the length of the polynomial expression.

\subsection{The rational Alexander module}\label{sec:ram}

The classical integral Alexander module $\azn{0}$ imbeds into the classical rational Alexander module $\an{0}$ by the map \[ \azn{0} = \gn{1}/\gn{2} \hookrightarrow \gn{1}/\gn{2} \otimes_\bbz \bbq = \an{0} \] which takes $a\mapsto a\otimes 1$. This map is an imbedding since $\gn{1}/\gn{2}$ is $\bbz$-torsion free~\cite[Theorem 1.3]{Crowell}.  The rational Alexander module $\an{0}$ is a module over $\bbq\left[t^{\pm1}\right]$, a PID, and so $\an{0} \cong \bigoplus_{i=1}^n \bbq\left[t^{\pm1}\right]/ \langle p_i(t) \rangle $ where $p_i(t) | p_{i+1}(t)$.  Let $d = \sum_{i=1}^n \deg(p_i(t))$.  One can easily construct a $\bbq\left[t^{\pm1}\right]$-module isomorphism \[ \bigoplus_{i=1}^n \bbq\left[t^{\pm1}\right]/ \langle p_i(t) \rangle \cong \bbq^d \] by recording the polynomial coefficients in the factors.  (There is an arithmetic formula for $t$-action on $\bbq^d$ which can be understood by the reducing of polynomials modulo the $p_i(t)$ on the left-hand side).

To sum up, there is an embedding of $\bbz\left[t^{\pm 1}\right]$-modules $\gn{1}/\gn{2} \hookrightarrow \bbq^d$, which induces an imbedding of rings \[ \bbz\left[\gn{1}/\gn{2}\right] \hookrightarrow \bbz[\bbq^d] \]  The set $\bbz[\bbq^d]$ is easy to work with.  Its elements are functions from $\bbq^d$ to the integers (with finite support), which can be implemented on a computer using dictionaries keyed by $d$-dimensional rational vectors and with values in the integers.

Since recognizing $\overrightarrow{0} \in \bbq^d$ is trivial, this method of tensoring $H = \gn{1}/\gn{2}$ with $\bbq$ gives a proof of Theorem~\ref{thm:swp} in the case $G$ is the fundamental group of a knot compelement.

From a computer programming point of view, this approach is untenable, insofar as computing $\delta_1(K)$ is concerned.  Encoding $\bbz\left[\gn{1}/\gn{2}\right]$ as a submodule of $\bbz[\bbq^d]$ is quite simple, but then arithmetic operations in $\bbq^d$ are required to upper-triangularize the metabelian Fox matrix.  To avoid numerical round-off errors, one must use arbitrary precision arithmetic in $\bbq^d$, which has substantial overhead.  Arbitrary precision arithmetic uses rational numbers with arbitrarily large numerators and denominators.  When a modern-day computer does an arithmetic operation an integer with greater than 20 digits ($20 \approx \log(2^{64})$), it splits the integer into shorter pieces, which uses memory and takes time.

We attempted implementing this `tensoring with $\bbq$' approach in step 4 of Algorithm~\ref{alg:outline}, and our program exceeded the computer's 8GB of RAM during computations for $\delta_1$ of several nine crossing knots.  Perhaps one explanation is that even relatively few additions in $\bbq$ can complicate the numerators and denominators: $\frac{1}{99} - \frac{1}{100} + \frac{1}{101} = \frac{10001}{999900}$.  The $t$-action also complicates the rational numbers (cf. Example~\ref{ex:rgex}).

\subsection{The Silver-Williams structure theorem}\label{sec:sw}

Given two maps $f:U\to A$, $g:U\to B$ in the category of abelian groups, one can form the \emph{amalgamated sum} \[ A \oplus_U B := A\oplus B / \{(f(u), 0) = (0,g(u))\mathrm{\ for\ all\ }u\in U\}\]  The amalgamated sum is a pushout in the category of abelian groups. If both $f$ and $g$ are injective, then $A$ and $B$ imbed into $A\oplus_U B$.  If one considers the case $A=B$, then one can iterate the operation.  If $A=B$ and $f$ and $g$ are injective, then this iteration is associative, i.e. \[ \left( B\oplus_U B \right) \oplus_U B \cong B\oplus_U \left( B\oplus_U B \right) \]  One can form infinite amalgamated sums as well: $\cdots\oplus_U B \oplus_U B\oplus_U \cdots$.

Given a finitely presented $\bbz[\bbz]$-module $H$ which is $\bbz$-torsion free, Silver and Williams~\cite{SilWil} construct explicit finitely generated abelian groups $U$ and $B$ and explicit maps $f,g:U\to B$.  By replacing $U$ with $U / (\ker f + \ker g)$ and $B$ with $B/ (g(\ker f) + f(\ker g))$ (and $f$ and $g$ by the induced maps) repeatedly until both $f$ and $g$ are injective.  This procedure terminates after finitely many steps by the Noetherian property of finitely generated abelian groups.  Silver and Williams construct a $\bbz[\bbz]$-module isomorphism $H \to \cdots\oplus_U B \oplus_U B\oplus_U \cdots$. The $t$-action on the amalgamated sum is given by shifting the $B$ factor to the right.

Using this isomorphism, we briefly describe an algorithm to decide whether a given element of $H$ is zero.  Given $x\in H$, apply the Silver-Williams isomorphism.  The image lies in a finite range of summands $C := B \oplus_U \cdots \oplus_U B$.  Since the amalgamating maps are injective, $ C \hookrightarrow \cdots\oplus_U B \oplus_U B\oplus_U \cdots$.  Deciding whether an element of $B \oplus_U \cdots \oplus_U B$ is zero is relatively simple.  Let $A = B \oplus_U \cdots \oplus_U B$, which has one fewer $B$ factor than $C$, so that $C \cong A \oplus_U B$.  Then an element $y \in C$ is zero if and only if it can be written as a sum of elements of the form $(f(u), -g(u))$, for $u\in U$.  In other words $y\in C$ is zero if and only if the $B$-component can be pulled back by $g$ and pushed into $A$ via $f$ so that the sum of the $A$-component with this `swept back' $B$-component is zero in $A$.  One applies this sweeping (left) procedure until $A=B$, a finitely generated abelian group, where the detection of $0$ is simple.  We summarize this in the following proposition, which with the above remarks gives a solution to the word problem for the Silver-Williams group presentation of $H$.

\begin{lem}\label{lem:sw}
	An element $y$ of $B \oplus_U \cdots \oplus_U B$ is zero if and only if $y$ can be `swept' to an element $y'$ in the left-most $B$-factor and this $y'=0$.
\end{lem}

While this approach allows one to detect $0$ in the ring $\bbz\left[\gn{1}/\gn{2}\right]$, it does not give a normal form for elements of $\gn{1}/\gn{2}$.  For example, an element in $C$ could be swept to the left out of $C$ -- there is no canonical $B$-factor to stop the sweeping.  Thus, it does not speed up the $O(k^2)$ algorithm for collapsing elements of $\bbz\left[\gn{1}/\gn{2}\right]$, and so we chose not to implement the Silver-Williams method to complete step 4 of Algorithm~\ref{alg:outline}.

\subsection{Gr\"obner bases}\label{sec:gb}

Gr\"obner bases over Laurent polynomial rings have been used for algorithmic computations of the elementary Alexander ideals~\cite{GVHHUE} but not, to the author's knowledge, to compute the entire Alexander module.

Let $R$ denote the polynomial ring $\bbz[r,s]$.  We will use the \emph{degree reverse lexicographic} order on monomials in $R$: $r^as^b < r^cs^d$ if and only $a+b < c+d$ or $\{$ $a+b = c+d$ and $d<b$ $\}\}$.  For example, $s < r < s^3 < rs^2 < r^2s < r^3$.  This is a total ordering on the monomials in $R$ so that any polynomial may be written in the unique way so that the monomials are increasing.  A \emph{Gr\"obner basis} $B$ for an ideal $I$ in $R$ is a finite,  list of polynomials $f_1,\ldots f_n$ so that division of any polynomial $R$ by the elements of $B$ gives a unique remainder.  Here, division of $f$ by $B$ means, roughly, to reduce $f$ modulo the elements of $B$ until it can be reduced no further.  More specifically, division of $f$ by $B$ is achieved by running through this loop: while there is an $f_i\in B$ such that $f = q\cdot f_i + g$ for some $q\neq 0 $ and $\deg(g) < \deg(f)$, set $f = g$ and repeat until $q=0$ for all choices of $i$.  The computations of the quotients $q$ of course depends on the choice of $f_i$ for each pass through the loop, but the output of the loop, i.e. the remainder of $f$ divided by $B$, is unique since $B$ is a Gr\"obner basis.  We denote $f \% I$ the reduction of $f\in R$ modulo the Gr\"obner basis $B$.  It follows that $f \% I = 0$ if and only if $f \in I$.

One can construct an ordering on $R^d$ from a monomial ordering in $R$ and a `vector ordering.'    Let $\overrightarrow{v}$ and $\overrightarrow{w}$ be vectors in $R^d$, and let $i$ (respectively, $j$) denote the index of the last nonzero entry of $\overrightarrow{v}$ (respectively, $\overrightarrow{w}$).  We say $\overrightarrow{v} < \overrightarrow{w}$ if either $i < j$ or $\{$ $i = j$ and $\overrightarrow{v}(i) < \overrightarrow{w}(i)$ $\}$.  This gives a total ordering on the elements of $R^d$.  Recall that the division algorithm in $R$ computes quotients by looking at the monomials in descending order.  This division algorithm can be extended to $R^d$ by defining a monomial in $R^d$ to be an $R$-monomial times a standard basis vector in $R^d$.

The notion of Gr\"obner bases have been extended to finitely presented $R$-modules.  We refer the reader to~\cite[Chapter 3]{AdamsLoustaunau} but discuss the salient points here.  Given a finite presentation matrix $P$ for an $R$-module $M$, one obtains an isomorphism $M\cong \mathrm{coker\ } P$.  Thus, the elements of $M$ may be thought of as vectors in $R^d$ modulo the column space of $P$.  A Gr\"obner basis for the module $M$ (with respect to the ordering on $R^d$) is a list $B$ of vectors $\overrightarrow{v_1},\ldots,\overrightarrow{v_n}$ such that the reduction of any $\overrightarrow{v}$ modulo $B$ produces a unique remainder.  We denote the reduction of $\overrightarrow{v}$ modulo the Gr\"obner basis $B$ by $\overrightarrow{v}\% B$.  Such a basis is useful because it gives a normal form for elements of $M$, i.e. $\overrightarrow{v} = \overrightarrow{w}$ as elements of $M$ if and only if $\overrightarrow{v} \% B = \overrightarrow{w} \% B$ in $R^d$.  In particular, a Gr\"obner basis affords one an easy method to detect $0\in M$, and such a basis is algorithmically computable.

We are interested not in the ring $R$ but in its quotient $\bbz[r,s]/\langle rs -1 \rangle \cong \bbz\left[t^{\pm 1}\right]$.  Given a presentation matrix $P$ for a $\bbz\left[t^{\pm 1}\right]$-module $M$, change all occurrences of $t$ to $r$ and $t\inv$ to $s$.  Then augment $P$ by the matrix $(rs-1) I$, and call the result $P'$.  Reducing a vector $\overrightarrow{v}$ modulo a Gr\"obner basis $B'$ for the module presented by $P'$ effectively reduces $\overrightarrow{v}$ by the columns of $P$ and sets $s = r\inv$.  In other words, for  any $\overrightarrow{v}\in R^d$, $\overrightarrow{v} \% B'$ is a normal form for the image of $\overrightarrow{v}$ in $M$.  We summarize this discussion in the following lemma.

\begin{lem}
	Let $H$ be a finitely presented $\bbz\left[t^{\pm 1}\right]$-module.  Then there exists a Gr\"obner basis $B$ for $H$ that provides a normal form $x \% B \in \bbz\left[t^{\pm 1}\right]^d$ for an element $x\in H$.  Furthermore, $x \% B = \overrightarrow{0}$ if and only if $x=0 \in H$.
\end{lem}

While polynomial division takes CPU time, this Gr\"obner basis method has provided a more effective implementation of Algorithm~\ref{alg:outline}'s step 4 than the `tensoring with $\bbq$' method.  While this method is slightly slower for the trefoil and figure-eight, it uses far less RAM and is able to calculate $\delta_1$ for the knots $11_{n67}$, $12_{n293}$, and the (at most) $23$ crossing knot from Example~\ref{ex:nut}.  We will illustrate the overhead benefits of the Gr\"obner basis method.

\begin{exam}\label{ex:rgex}
	Let $P$ be the matrix \[ P = \begin{pmatrix}
		3 - 3t & 2-t\\
		1-2t & 3-3t
	\end{pmatrix} \] as in Example~\ref{ex:nut}.  Let us view $P: \bbz[\bbz]^2\to\bbz[\bbz]^2$ by $\overrightarrow{x}\mapsto \overrightarrow{x}P$.  So the module $P$ presents is $\bbz[\bbz]^2$ modulo the row space of $P$.  This agrees with the convention for $\azn{0}$ as established in Example~\ref{ex:tref}.  We aim to find the normal form of the vectors $\overrightarrow{x} = (3t^3 + 2t^2 + 2)e_1 + (5t^2 - t^3) e_2$ and $t*\overrightarrow{x}$ using the rational Alexander module and Gr\"obner basis methods.
	
	Working over $\bbq[\bbz]$, we calculate the Smith normal form of $P$ to be $D = LPR$, where \[ D = \begin{pmatrix}
		1&0\\
		0&7t^2-13t+7
	\end{pmatrix} \mathrm{\ and\ } R = \begin{pmatrix}
		1 & -7/3t + 5/3\\
		0&1
	\end{pmatrix} \]  For this example, $L$ is irrelevant but may be computed using inverses.  As a $\bbq[\bbz]$-module, $M\otimes \bbq$ is $\bbq[\bbz]/\langle 7t^2-13t+7 \rangle$, though recall that $M$ merely injects into $M\otimes \bbq$ as a $\bbz[\bbz]$-module.  This quotient ring is isomorphic to $\bbq^2$ by recording the coefficients of the linear and constant terms of a polynomial's reduction modulo $7t^2 - 13t+7$.
	
	Now given any element $y \in \bbz[\bbz]^2$ which projects to an element $\overline y \in M$, we have that $\phi(\overline y) = \overline{yR}$, where $\phi$ is the map that takes $M$ to $\bbq[\bbz]/\langle 7t^2-13t+7 \rangle$ and $\overline{\phantom{x}} $ denotes taking the quotient modulo $P$ or $D$.  After composing with the isomorphism to $\bbq^2$, we see that our $\overrightarrow{x} \mapsto \left(\frac{-1420}{147}, \frac{281}{21}\right)$ and $t\overrightarrow{x} \mapsto \left(\frac{-4691}{1029}, \frac{1420}{147}\right)$.  Notice how the $t$-action complicates the rational entries.  Addition in this module complicates these rational numbers even more.
	
	Let $r = t$ and $s = t\inv$.  Using Macaulay2~\cite{M2}, we compute a Gr\"obner basis as described above to be \[ G = \left\{\begin{pmatrix}
		7r+7s-13\\0 
	\end{pmatrix}, 	\begin{pmatrix}
			7s^2-13s+7\\ 0
		\end{pmatrix}, 	\begin{pmatrix}
				rs-1\\ 0
			\end{pmatrix}, 	\begin{pmatrix}
					7s-5\\ 3
				\end{pmatrix}, 	\begin{pmatrix}
						5s-4\\s+1 
					\end{pmatrix}, 	\begin{pmatrix}
							3r+7s-8\\ r+1
						\end{pmatrix}\right\}
	\]  Macaulay2 computes $\overrightarrow{x}\% G = (-r^3+3r^2-2r-14s+14, 0)$ and $(r\overrightarrow{x})\% G = (-r^4+3r^3-2r^2-14s+12, 0)$.
	
	In the $\otimes \bbq$ method, the addition and $t$-action have inherent definitions in $\bbq^2$.  This is easy to implement and feels like it should be fast.  However, the rational numbers become complicated after a few operations. Each subsequent operation takes increasingly more time, as evidenced by Example~\ref{ex:qbad}.  In practice, this method's implementation exhausts our machine's memory during even simple $\delta_1$ calculations.
	
	In the Gr\"obner basis method, the addition and $t$-action do not have inherent definitions, i.e. $(\overrightarrow{x}+\overrightarrow{y})\%G \neq \overrightarrow{x}\% G + \overrightarrow{y}\% G$ and $(t\overrightarrow{x})\%G \neq t(\overrightarrow{x}\%G)$.  After every addition and $t$-action, one must divide the result by the Gr\"obner basis, which takes CPU time.  However, the form $\overrightarrow{x}\% G$ is simpler -- requiring less RAM  than its $\bbq^2$ counterpart -- and we find the Gr\"obner basis implementation of encoding the ring $\bbz\left[ \gn{1}/\gn{2} \right]$ more effective.
	
\end{exam}

\begin{exam}\label{ex:qbad}
	The $t$-action on $\bbq^2$ in the above example is given by $t*(a,b) = \left(b + \frac{13}{7}a, -a\right)$.  Consider the elements $v_0 = (-1420/147, 281/21)$ and $x=(1282/147,-167/21)$ of $\bbq^2$.  These are the images under the isomorphism $\bbz[\bbz]^2 / \mbox{row }P \to \bbq^2$ of the elements $(3t^3 + 2t^2 + 2, 5t^2 - t^3)$ and $(t^2-3t+7, 4t^2 + 5t^3-2)$, respectively.  Define a sequence $v_i$ by $v_{n} = t*v_{n-1} + x$.
	
	In the algorithm to compute $\delta_1$, one must perform row operations on a matrix whose entries lie in $\bbk_1\left[t^{\pm 1}\right]$.  The coefficient of a monomial of one entry of this matrix is a formal quotient of elements of $\bbz\left[ \gn{1}/\gn{2} \right]$.  The numerator and denominator of this quotient may be quite long.  Each term in the numerator is an integer times an element of $\gn{1}/\gn{2}$.  If we think of $\gn{1}/\gn{2}$ as a subset of $\bbq^2$, it is tempting to absorb this integer into the vector, but this is not permissible in the group ring $\bbz[\bbq^2]$.  A single row operation on this matrix necessitates adding many monomials and hence many coefficients in $\bbk_1$.  Each addition $a/b + c/d = (ad+bc)/bd$ in $\bbk_1$ involves three multiplications in the group ring $\bbz[\bbq^2]$.  Each group ring multiplication involves \emph{many} group operations in $\bbq^2$, depending on the number of terms in the $\bbz[\bbq^2]$-elements which are multiplied.  Note that each row operation on the matrix vastly increases the number of group operations in $\bbq^2$ that must be done.  After a few row operations, there are a multitude of $\gn{1}/\gn{2}$ elements involved.  If we think of these as elements of $\bbq^2$, the numerators and denominators of the entries of each vector in $\bbq^2$ grow longer as the number of $\bbq^2$ operations increases.  This, we believe, is why the `tensoring with $\bbq$' implementation of step 4 exhausts our computer of all available memory.
	
	The sequence $v_n$ is intended to indicate how complicated the rational numbers become after many arithmetic operations.  We do not foresee the $v_n$ arising in a $\delta_1$ computation.  While the rational model for $\gn{1}/\gn{2}$ offers a faster implementation for the $\bbz\left[ \gn{1}/\gn{2} \right]$-operations than the Gr\"obner basis model -- at least when few operations are required -- the rational model slows down significantly when many operations are performed.  We computed $v_n$ to $n=50000$ and noticed a significant slowing down of the computation as $n$ increased.  See the table in Figure~\ref{fig:table}.  The CPU time required to compute $v_n$ appears to be exponential in $n$.
	
	\begin{figure}[!ht]
		\begin{center}
			\begin{tabular}{ c|c|c }
			  Step $i$ & \parbox{4cm}{CPU time to compute $v_{1000i}$, in seconds} & \parbox{4cm}{Digits in denominator of first entry of $v_{1000i}$} \\ \hline
			1 & 0.048321 & 844\\
			2 & 0.146888 & 1689\\
			3 & 0.315882 & 2534\\
			4 & 0.57017 & 3380\\
			5 & 0.921086 & 4225\\
			10 & 4.870255 & 8450\\
			20 & 28.885591 & 16901\\
			30 & 84.075005 & 25352\\
			40 & 181.961961 & 33803\\
			45 & 249.442993 & 38029\\
			46 & 265.189638 & 38874\\
			47 & 281.562427 & 39719\\
			48 & 298.729361 & 40564\\
			49 & 316.309955 & 41409\\
			50 & 334.467209 & 42254
			\end{tabular}
			\label{fig:table}
		\end{center}
	\end{figure}
	
	Note that the terms of this sequence $v_n$ written in the Gr\"obner basis model will also become complicated as $n$ increases.  By the Gr\"obner basis in Example~\ref{ex:rgex}, the second entry in $v_n$ will be an integer, and the first entry will be a polynomial in $r$ (with perhaps one or two $s$ terms) whose degree grows sublinearly in $n$ and most (all but two) of whose coefficents have absolute value less than $7/2$ (Macaulay2 computes a quotient so that the remainder's leading coefficient is as small in absolute value as possible, i.e. $13r\  \%\  5r = -2r$).  Note that the second entry has absolute value less than $3/2$.  Thus, $v_n$ can be stored in a computer by approximately $n$ digits.  This is more efficient than the rational method, at least experimentally, for approximately $n$ digits were used \emph{just for the denominator of one entry} in the $\bbq^2$ representation of $v_n$, according to the table above.
\end{exam}

\section{Computations}

Using a 2.5 GhZ Intel i5 processor with four cores and 8 GB of RAM, we made the following $\delta_1$ calculations.  The implementation is discussed in Section~\ref{sec:implementation}.  We obtained the PD code for all knots using KnotInfo~\cite{KnotInfo}.  For the $23$ crossing Example~\ref{ex:nut2}, we used the program Link Sketcher, which we found on KnotInfo, to compute the PD Code.

\begin{exam}
	The trefoil knot $3_1$ has PD code $X[1, 5, 2, 4], X[3, 1, 4, 6], X[5, 3, 6, 2]]$.  From this PD code, our program calculated $\delta_1(3_1) = 1$ in $0.2128$ seconds.
\end{exam}

\begin{exam}
	The figure-eight knot $4_1$ has PD code $X[4, 2, 5, 1], X[8, 6, 1, 5], X[6, 3, 7, 4], X[2, 7, 3, 8]$.  Our program calculated $\delta_1(4_1) = 1$ in $0.5616$ seconds.
\end{exam}

\begin{exam}
	The $2,5$-torus knot $5_1$ has PD code $X[2, 8, 3, 7], X[4, 10, 5, 9], X[6, 2, 7, 1], X[8, 4, 9, 3], X[10, 6, 1, 5]$.  Our program calculated $\delta_1(5_1)=3$ in $9.5816$ seconds.
\end{exam}

\begin{exam}
	The knot $5_2$ has PD code $X[1, 5, 2, 4], X[3, 9, 4, 8], X[5, 1, 6, 10], X[7, 3, 8, 2], X[9, 7, 10, 6]$.  Our program calculated $\delta_1(5_2)=1$ in $1.4597$ seconds.
\end{exam}

\begin{exam}
	The stevedore knot $6_1$ has PD code $X[1, 7, 2, 6], X[3, 10, 4, 11], X[5, 3, 6, 2], X[7, 1, 8, 12], \\ X[9, 4, 10, 5], X[11, 9, 12, 8]$.  Our program calculated $\delta_1(6_1) = 1$ in $2.7247$ seconds.
\end{exam}

\begin{exam}
	The knot $11_{n67}$ has PD code $X[4, 2, 5, 1], X[8, 4, 9, 3], X[11, 17, 12, 16], X[14, 5, 15, 6],\\ X[6, 15, 7, 16], X[9, 19, 10, 18], X[17, 11, 18, 10], X[19, 1, 20, 22], X[13, 20, 14, 21], X[21, 12, 22, 13], X[2, 8, 3, 7]$.  This PD code produces a Wirtinger presentation with $11$ generators, and our program would attempt to put a $10\times 10$ matrix over $\bbk_1\left[t^{\pm 1}\right]$ in upper-triangular form.  However, $11_{n67}$ is a $3$-bridge knot, and thus its fundamental group has a presentation with $3$ generators and $2$ relations.  Using Tietze transformations in GAP~\cite{GAP}, we were able to reduce the Wirtinger presentation to the generators $x_1,x_2,x_3$ and relations $x_1^{-1} x_2 x_3^{-1} x_2^{-2} x_1 x_3 x_2 x_1^{-1} x_2 x_3^{-1} x_2^{-2} x_1 x_2^{-1} x_3^{-1} x_1^{-1} x_2^2 x_3 x_2^{-1} x_1 x_2 x_3 x_2^{-1}$ and $x_1 x_3^{-1} x_1^{-1} x_2 x_3^{-1} x_2^{-2} x_3^{-1} x_1^{-1} x_2^2 x_3 x_2^{-1} x_1^2 x_3 x_1^{-1} x_2 x_3^{-1} x_2^{-2} x_1 x_3 x_2^2 x_3 x_2^{-1} $.  In this presentation, each $x_i$ maps to a generator of $\bbz$, and so a presentation with a `nice' (as in Example~\ref{ex:tref}) generating set can easily be obtained.   This results in a $2\times 2$ metabelian Fox matrix that must be put in upper triangular form.  Of course, the entries in the $2\times 2$ matrix are more complicated than in the original $10\times 10$ matrix, but the $\delta_1$ computation is faster.
	
	Given the reduced Wirtinger presentation, our program calculated $\delta_1(11_{n67}) = 3$ in $11324$ seconds.  As mentioned around Theorem~\ref{thm:11n67}, this is the first known knot in the tables with $\delta_1 > \delta_0$.
\end{exam}

\begin{exam}
	The knot $12_{n293}$ has PD code $X[1, 4, 2, 5], X[3, 10, 4, 11], X[5, 12, 6, 13], X[16, 8, 17, 7], \\ X[9, 2, 10, 3], X[11, 8, 12,9], X[20, 13, 21, 14], X[6, 16, 7, 15], X[24, 17, 1, 18], X[22, 19, 23, 20], X[14, 21, 15, 22], \\ X[18, 23,19, 24]$.  Using GAP, we found a simplified presentation for the fundamental group of the knot complement with generators $a,b,c$ and relators $b^2ab^{-2}a^{-1}c^{-1}bacab^2a^{-1}b^{-2}a^{-1}c^{-1}a^{-1}b^{-1}cab$ and $b^{-1}a^{-1}c^{-1}bab^{-1}cabcab^{-1}a^{-1}c^{-1}bab^{-1}a^{-1}c^{-1}ba^{-1}b^{-1}caba^{-1}c^{-1}$.  Here $a$ maps to a generator under abelianization, and $b$ and $c$ map to zero.  Our program calculated $\delta_1(12_{n293}) = 3$ in 8393 seconds.  This is another example of a knot with $2\cdot\mathrm{genus} = 1+\delta_1 > 2\cdot\delta_0$.
\end{exam}

\begin{exam}\label{ex:nut2}
	Let $K$ denote the knot from Example~\ref{ex:nut}.  The diagram for $K$ in Figure~\ref{fig:nut} has $23$ crossings, and its PD code is $X[34,1,35,2], X[2,33,3,34], X[32,3,33,4], X[4,31,5,32], X[30,5,31,6], X[6,29,7,30], \\ X[7,40,8,41], X[8,20,9,19], X[22,10,23,9], X[37,10,38,11], X[11,36,12,37], X[12,24,13,23], \\ X[46,13,1,14], X[14,45,15,46], X[44,15,45,16], X[16,43,17,44], X[42,17,43,18], X[18,41,19,42], \\ X[27,20,28,21], X[21,26,22,27], X[24,36,25,35], X[38,26,39,25], X[28,40,29,39]$.  Using Tietze transformations in GAP, we were able to reduce the Wirtinger presentation with three generators $x_1,x_2,x_3$ and three relations $x_1 x_2\inv x_1 x_2 x_1\inv x_2 x_3\inv x_2 x_3\inv x_2\inv x_3 x_2\inv$ and $x_3\inv x_1 x_2\inv x_1 x_2 x_1\inv x_2 x_1\inv x_3 x_1\inv x_3\inv x_1$.
	
	Our program calculated $\delta_1(K) = 1$ in $0.7698$ seconds.
\end{exam}

\section{Mutant examples: the proof of Theorem~\ref{thm:DetectsMutants}}\label{sec:Mutants}

Figure~\ref{fig:12mutants} shows the knots $12_{n23}$ and $12_{n31}$. One can check the two are related by rotating the contents of the dashed circle by $\pi$, hence the knots $12_{n23}$ and $12_{n31}$ form a mutant pair.  This agrees with~\cite{DWL}.

\begin{figure}[!ht]
	\begin{center}
		\begin{overpic}[scale=.45]
			{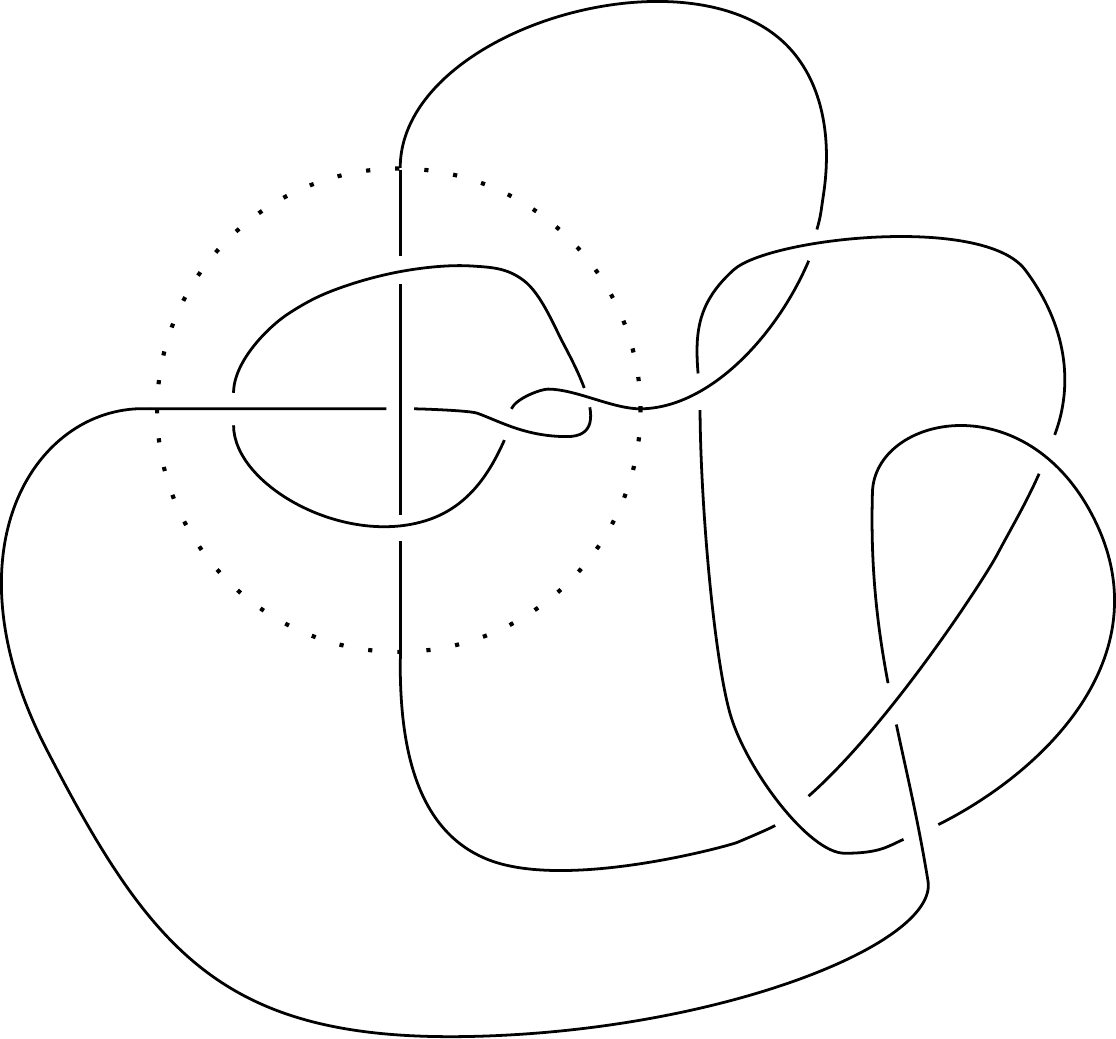}
		\end{overpic}
		\hspace{1cm}
		\begin{overpic}[scale=.45]
			{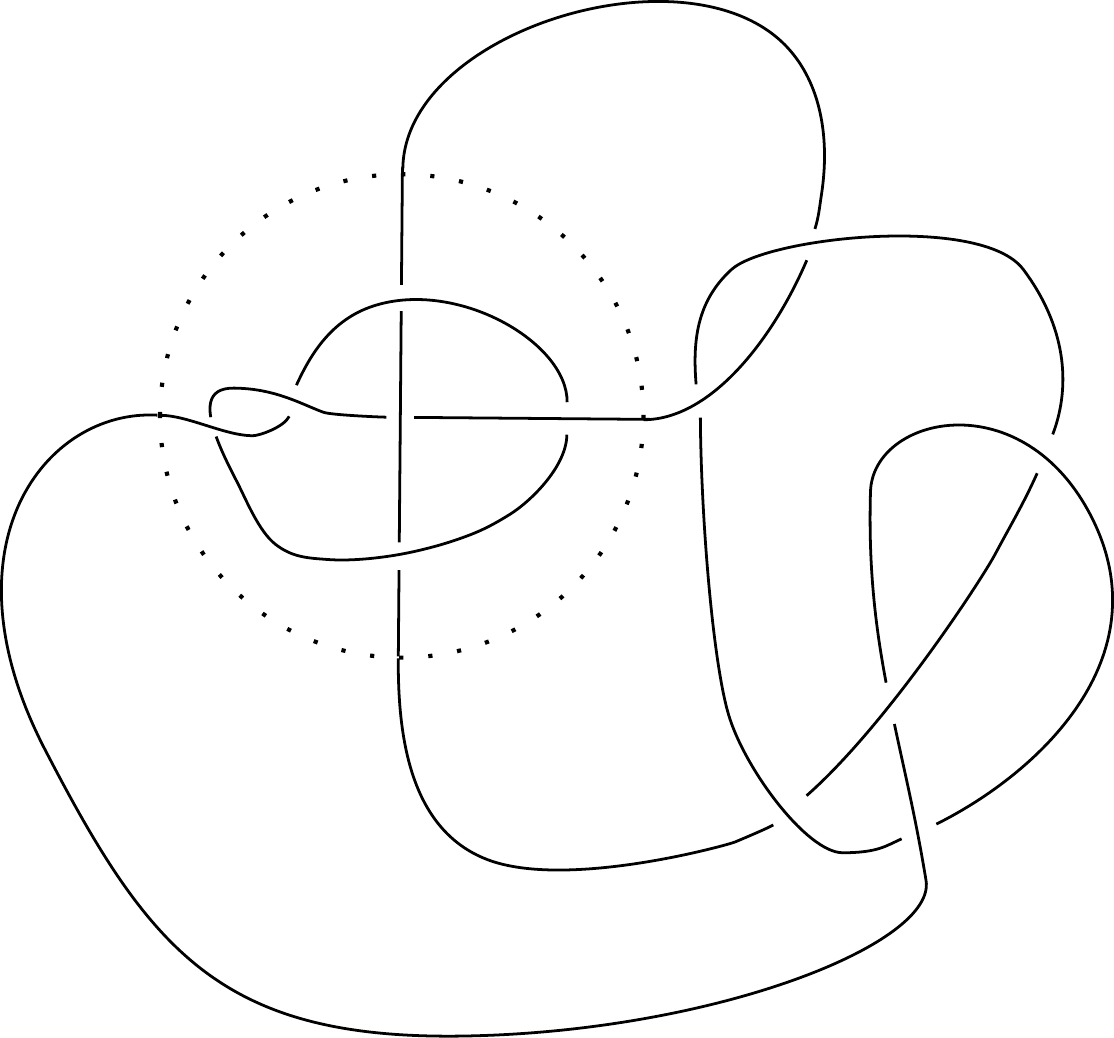}
		\end{overpic}
	\end{center}
	\caption{$12_{n23}$ and $12_{n31}$}
	\label{fig:12mutants}
\end{figure}

\begin{exam}\label{ex:12n23}
	We have found by experiment that a presentation for the knot group coming from an arc diagram (sometimes called a grid diagram or $\ast$ projection) is a suitable input for our $\delta_1$ program.  Such a presentation can sometimes yield a much quicker $\delta_1$-computation than a Wirtinger presentation.  We will quickly review arc diagrams for a knot.  An arc diagram for a knot is a planar diagram for a knot that can be arranged in an $n\times n$ square grid.  Each segment of the knot is either vertical or horizontal, and each row (respectively, column) of the grid contain precisely one horizontal (respectively, vertical) segment of the knot.  We use the convention that vertical segments must pass over horizontal segments.  Every knot has such a diagram~\cite[p. 205]{Neuwirth}, and from such a diagram, one can compute a presentation for the fundamental group of the complement of the knot~\cite[Section 5]{Neuwirth},~\cite[Proof of Proposition 6.2]{MOST}.
	
	To get a presentation for the knot group from an arc diagram (that is compatible with our program), first orient the knot one way or the other.  The generating set for the group will be $x_1,\ldots, x_n$, where $n$ is the size of the grid.  The generators correspond to (positive) meridians of the (oriented) vertical segments of the knot projection, numbered from left to right.  The knot necessarily crosses $n-1$ of the horizontal gridlines in the $n \times n$ grid.  For each such horizontal gridline, form a group word by recording from left to right the vertical segments crossed by that gridline, in such a way that the upward segments are recorded to the first power and downward segments to the negative first power; these words form the relators for the fundamental group.  For this presentation to be compatible with our program, one must ensure each generator maps to the preferred meridian of the knot.  This is accomplished by the aforementioned orientation conventions.
	
	We obtained the arc diagram in Figure~\ref{fig:gridsofmutants} for $12_{n23}$ from the program Gridlink~\cite{Culler}, which generates the arc diagram from the braid data for knots in the repository KnotInfo~\cite{KnotInfo} (and then simplifies the diagram).  Using this arc diagram, we computed the knot group's presentation as described above.  That presentation reduces (via Tietze transformations in GAP~\cite{GAP}) to the presentation with generators $x_1, x_2, x_3, x_4$ and relators:
	\[x_4 x_2^{-1} x_4^{-1} x_1 x_2^{-1} x_4 x_2 x_1^{-1},\]
	\[x_3 x_1^{-1} x_3 x_1 x_3^{-1} x_4 x_1^{-1} x_3 x_1^{-1} x_3^{-1} x_1 x_4^{-1}, \mathrm{and}\]
	\[x_1^{-1} x_3 x_1 x_3^{-1} x_1 x_2^{-1} x_4^{-1} x_2 x_4 x_2 x_1^{-1} x_3 x_1^{-1} x_3^{-1} x_1 x_3 x_1^{-1} x_3 x_1 x_3^{-1} x_4 x_2^{-1} x_4^{-1} x_2^{-1} x_4 x_2 x_1^{-1} x_3 x_1^{-1} x_3^{-1} x_1 x_3^{-1}\]
	
	Our program computed $\delta_1(12_{n23}) = 3$ in $3435$ seconds.
\end{exam}

\begin{figure}[!ht]
	\begin{center}
		\begin{overpic}[scale=.45]
			{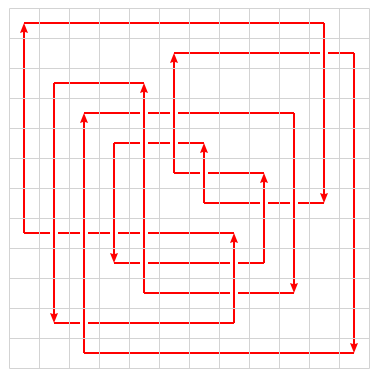}
		\end{overpic}
		\hspace{1cm}
		\begin{overpic}[scale=.45]
			{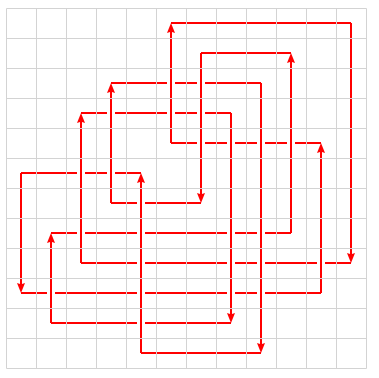}
		\end{overpic}
	\end{center}
	\caption{Arc diagrams for $12_{n23}$ and $12_{n31}$}
	\label{fig:gridsofmutants}
\end{figure}

\begin{exam}\label{ex:12n31}
	Using the arc diagram in Figure~\ref{fig:gridsofmutants} for $12_{n31}$, we found the reduced presentation with generators $x_1, x_2, x_3, x_4$ and relators:
	\[x_4 x_2 x_4^{-1} x_1 x_2^{-1} x_4^{-1} x_2 x_1^{-1},\]
	\[x_3^{-1} x_2 x_1^{-1} x_3 x_1 x_2^{-1} x_3 x_2 x_1^{-1} x_3^{-1} x_1 x_2^{-1} x_1^{-1} x_3 x_1 x_2^{-1}, \mathrm{and}\]
	\[x_3 x_4 x_2 x_4^{-1} x_2^{-1} x_4^{-1} x_2 x_1^{-1} x_3^{-1} x_1 x_3 x_4 x_2 x_4 x_2^{-1} x_4^{-1} x_3^{-1} x_1^{-1}\]
	Our program computed $\delta_1(12_{n31}) = 5$ in $7256$ seconds.
\end{exam}

\section{An implementation of the algorithm}\label{sec:implementation}

We implemented Algorithm~\ref{alg:outline} in Sage 5.3.  Most of the classes and functions are written in Python 2.7.2, and we used Macaulay2 version 1.4 to compute the Gr\"obner basis for $\azn{0}$.  Sage incorporates Python and Macaulay2, and so our program can be executed in Sage with one click of a button.  Sage~\cite{Sage}, Python~\cite{Python}, and Macaulay2~\cite{M2} are open source and cross-platform.

Our implementation can take as an input the PD code for a knot diagram or a Wirtinger presentation for the knot complement's fundamental group.  We use a script written by REU students to convert the PD code to a Wirtinger presentation~\cite{REU}.  At this point, the algorithm runs quite similarly to Example~\ref{ex:tref}.  After switching to a convenient generating set for the fundamental group, the Fox matrix is then calculated, and a presentation matrix (the abelian Fox Matrix) for $\gn{1}/\gn{2} \cong \azn{0}$ is recorded.  The abelian Fox matrix is fed to Macaulay2, which computes a Gr\"obner basis for the classical Alexander module.  Separately, the entries of the original Fox matrix are written modulo $\gn{2}$, using a splitting which treats the first generator of the Wirtinger presentation as $t$, so $G/\gn{2} \cong \gn{1}/\gn{2} \rtimes \bbz$, as in equation~\ref{eq:trefmat}.  Our program finds a column to delete by reducing the $\gn{1}/\gn{2}$-generators modulo the Gr\"obner basis computed earlier, and deletes it, yielding the metabelian Fox matrix, which presents $\an{1}$ as a $\bbk_1\left[t^{\pm 1}\right]$-module.  Finally, this metabelian Fox matrix is put in upper-triangular form by row and column operations, and the degrees of the polynomials are computed and added, giving $\delta_1(K)$.

We wrote classes for the elements of $\pi_1$, $G/\gn{2}$, $\bbz\Gamma_1$, $\bbk_1$, and $\bbk_1\left[t^{\pm 1}\right]$, and functions to add and/or multiply elements in each of those sets, which culminated in the polynomial ring operations.  Throughout the upper-triangularization process, we need the degree function in $\bbk_1\left[t^{\pm 1}\right]$, which amounts to recognizing when the coefficient of a term is zero or not, hence the importance of Section~\ref{sec:wp}.  In $\bbk_1\left[t^{\pm 1}\right]$, we need to compute quotients of polynomials on the left and right so that we can do row and column operations to upper-triangularize the metabelian Fox matrix.

We implemented group ring elements in $\bbz\left[\gn{1}/\gn{2}\right]$ as Python dictionaries.  A \emph{dictionary} is a function whose domain consists of \emph{keys}.  The range of the function is the set of \emph{values} for the dictionary.  In Python, keys must be hashable, so that dictionary lookup time (recalling the value of the function given a key) is constant time.  A group ring element in $\bbz\left[\gn{1}/\gn{2}\right]$ can then be thought of as a dictionary keyed by elements of $\gn{1}/\gn{2}$ with values in $\bbz$.  Multiplication of elements in $\bbz\left[\gn{1}/\gn{2}\right]$ follows the rule $(ax + by)(cz+dw) = ab(x+z) + ad(x+w) + bc(y+z) +bd(y+w)$ (we use here that the group $\gn{1}/\gn{2}$ is abelian).  Python dictionaries can be used to multiply and group like terms, since we have a normal form for elements of $\gn{1}/\gn{2}$.  We include the pseudo-code for our group ring multiplication function:

\begin{verbatim}
    def multiply_dictionaries(dict1, dict2):
        product = {};
        for v in keys of dict1:
            for w in keys of dict2:
                # v and w are vectors reduced by the Groebner basis GB;
                product_key = (v+w) % GB;
                product[product_key] = 0;
        for v in keys of dict1:
            for w in keys of dict2:
                product[(v+w) % GB] = product[(v+w) % GB] + dict1[v]*dict2[w];
        for  k in keys of product:
            if product[k] == 0:
                delete product[k];
        return product
\end{verbatim}

The first for loop creates all of the sums of group elements that occur in the product in the group ring.  The second for loop computes all the integer coefficients in the product. The third for loop removes all terms in the product with coefficient zero.  If $l$ is the length of the first factor and $r$ is the length of the second factor, this function runs through $3lr$ loops, at most.  The returned product is completely collapsed.  Alternatively, one could implement group ring elements as lists.  Multiplication and collapsing of lists must run through at most $lr + (lr)^2$ loops.

A Sage worksheet for computing $\delta_1$ is available at \url{https://pdhorn.expressions.syr.edu/software/}.

\bibliographystyle{amsalpha}
\bibliography{/Users/pdhorn/Documents/Research/PeterHornBib}

\newcommand{\etalchar}[1]{$^{#1}$}
\providecommand{\bysame}{\leavevmode\hbox to3em{\hrulefill}\thinspace}
\providecommand{\MR}{\relax\ifhmode\unskip\space\fi MR }
\providecommand{\MRhref}[2]{%
  \href{http://www.ams.org/mathscinet-getitem?mr=#1}{#2}
}
\providecommand{\href}[2]{#2}
\begin{thebibliography}{GVHHUE06}

\bibitem[AL94]{AdamsLoustaunau}
William~W. Adams and Philippe~I. Loustaunau, \emph{An introduction to
  {G}r\"obner bases}, Graduate Studies in Mathematics, vol.~3, American
  Mathematical Society, Providence, RI, 1994.

\bibitem[CL13]{KnotInfo}
Jae~Choon Cha and Charles Livingston, \emph{Knotinfo: Table of knot
  invariants}, \texttt{http://www.indiana.edu/$\sim$knotinfo}, April 2013.

\bibitem[Coc04]{C:NonCom}
Tim~D Cochran, \emph{Noncommutative knot theory}, Algebr. Geom. Topol.
  \textbf{4} (2004), 347--398.

\bibitem[Coh85]{Cohn}
Paul~{M} Cohn, \emph{Free rings and their relations}, 2 ed., London
  Mathematical Society Monographs, vol.~19, Academic {C}ourt [{H}arcourt
  {B}race {J}avonovich {P}ublishers], London, 1985.

\bibitem[COT03]{COT1}
Tim~D Cochran, Kent~E Orr, and Peter Teichner, \emph{Knot concordance,
  {W}hitney towers and ${L}^2$-signatures}, Ann. of Math. (2) \textbf{157}
  (2003), no.~2, 433--519.

\bibitem[Cro63]{Crowell}
Richard~H Crowell, \emph{The group ${G'/G''}$ of a knot group ${G}$}, Duke
  Math. J. \textbf{30} (1963), no.~349--354.

\bibitem[Cul]{Culler}
Marc Culler, \emph{Gridlink 2.0}, software, available at
  \texttt{http://www.math.uic.edu/}$\sim$\texttt{culler/gridlink}.

\bibitem[DL07]{DWL}
David {De Wit} and Jon Links, \emph{Where the {L}inks-{G}ould invariant first
  fails to distinguish nonmutant prime knots}, Journal of Knot Theory and its
  Ramifications \textbf{16} (2007), no.~8, 1021--1041.

\bibitem[Fox53]{Fox}
Ralph~H. Fox, \emph{Free differential calculus. {I}. {D}erivation in the free
  group ring.}, Annals of Mathematics \textbf{57} (1953), no.~3, 547--560.

\bibitem[FV11]{FV:Survey}
Stefan Friedl and Stefano Vidussi, \emph{A survey of twisted alexander
  polynomials}, Contributions in Mathematical and Computational Sciences,
  vol.~1, pp.~45--94, Springer, Heidelberg, 2011.

\bibitem[GAP12]{GAP}
The GAP-Group, \emph{{GAP -- Groups, Algorithms, and Programming, Version
  4.5.7}}, 2012, \href{http://www.gap-system.org/}{http://www.gap-system.org/}.

\bibitem[GS]{M2}
Daniel~R. Grayson and Michael~E. Stillman, \emph{Macaulay2, a software system
  for research in algebraic geometry}, Available at
  \href{http://www.math.uiuc.edu/Macaulay2/}%
  {http://www.math.uiuc.edu/Macaulay2/}.

\bibitem[GVHHUE06]{GVHHUE}
Jes\'us Gago-Vargas, Isabel Hartillo-Hermoso, and Jos\'e~Mar\'ia
  Ucha-Enr\'iquez, \emph{Algorithmic invariants for {A}lexander modules},
  Computer algebra in scientific computing (Victor~G. Ganzha, Ernst~W. Mayr,
  and Evgenii~V. Vorozhtsov, eds.), Lecture Notes in Computer Science, vol.
  4194, Springer, Berlin, 2006, pp.~149--154.

\bibitem[Har05]{Harvey:Thurston}
Shelly~L Harvey, \emph{Higher-order polynomial invariants of $3$-manifolds
  giving lower bounds for the {T}hurston norm}, Topology \textbf{44} (2005),
  895--945.

\bibitem[Hol10]{Holum}
Erik Holum, \emph{Calculating the degree of higher order {A}lexander
  polynomials}, Undergraduate honors thesis, Wesleyan University, Middletown,
  CT, 2010.

\bibitem[Lew72]{Lewin}
Jacques Lewin, \emph{A note on zero divisors in group-rings}, Proc. Amer. Math.
  Soc. \textbf{31} (1972), no.~2, 357--359.

\bibitem[LM87]{LM}
W.~B.~R. Lickorish and Kenneth~C. Millett, \emph{A polynomial invariant of
  oriented links}, Topology \textbf{26} (1987), no.~1, 107--141.

\bibitem[LM08]{LeiMax}
Constance Leidy and Laurentiu Maxim, \emph{Obstructions on fundamental groups
  of plane curve complements}, Real and complex singularities, Contemporary
  Mathematics, vol. 459, American Mathematical Society, Providence, RI, 2008,
  pp.~117--130.

\bibitem[MOST07]{MOST}
Ciprian Manolescu, Peter Ozsv\'ath, Zolt\'an Szab\'o, and Dylan Thurston,
  \emph{On combinatorial link {F}loer homology}, Geom. Topol. \textbf{11}
  (2007), 2339--2412.

\bibitem[Neu84]{Neuwirth}
Lee~P. Neuwirth, \emph{$\ast$ projections of knots}, Algebraic and differential
  topology--global differential geometry, Teubner-Texte zur Mathematik,
  vol.~70, Teubner, Leipzig, 1984, pp.~198--205.

\bibitem[Pyt12]{Python}
The Python Software Foundation, \emph{{P}ython, {V}ersion 2.7.2}, 2012,
  \href{http://www.python.org/}{http://www.python.org/}.

\bibitem[{REU}11]{REU}
{REU Group in Noncommutative Knot Theory}, \emph{{PDcodetoWirtGUI.py}},
  Columbia University Topology RTG, 2011,
  \href{http://pdhorn.expressions.syr.edu/}{http://pdhorn.expressions.syr.edu/}.

\bibitem[Rol76]{Rolf}
Dale Rolfsen, \emph{Knots and links}, Publish or Perish, Berkeley, CA, 1976.

\bibitem[S{\etalchar{+}}12]{Sage}
William~A. Stein et~al., \emph{{S}age {M}athematics {S}oftware, {V}ersion 5.3},
  The Sage Development Team, 2012,
  \href{http://www.sagemath.org/}{http://www.sagemath.org/}.

\bibitem[Sti82]{Stillwell}
John Stillwell, \emph{The word problem and the isomorphism problem for groups},
  Bull. Amer. Math. Soc. (New Series) \textbf{6} (1982), no.~1, 33--56.

\bibitem[Str74]{Strebel}
Ralph Strebel, \emph{Homological methods applied to the derived series of
  groups}, Comment. Math. Helv. \textbf{49} (1974), 302--332.

\bibitem[SW12]{SilWil}
Daniel Silver and Susan Williams, \emph{On modules over {L}aurent polynomial
  rings}, Journal of Knot Theory and its Ramifications \textbf{21} (2012),
  no.~1, 6 pages.

\bibitem[Wad94]{Wada}
Masaaki Wada, \emph{Twisted {A}lexander polynomial for finitely presentable
  groups}, Topology \textbf{33} (1994), no.~2, 241--256.

\bibitem[Wal78]{Wald}
Friedheml Waldhausen, \emph{Recent results on sufficiently large 3-manifolds},
  Algebraic and geometric topology (Proc. Sympos. Pure Math., Stanford Univ.,
  Stanford, CA, 1976), Part 2, Proceedings of Symposia in Pure Mathematics,
  vol. XXXII, American Mathematical Society, Providence, RI, 1978, pp.~21--38.

\end{thebibliography}
\end{document}